\documentclass{article}

\usepackage{amsmath, amssymb, amsthm, mathrsfs}
\usepackage{indentfirst}
\usepackage{graphicx, tikz}
\usepackage{geometry}
\usepackage{caption}
\usepackage{subfigure, booktabs}

\usepackage{hyperref}

\usepackage{cleveref}
\crefname{section}{Section}{Sections}
\crefname{figure}{Figure}{Figures}
\crefname{table}{Table}{Tables}
\crefname{equation}{}{}
\crefname{theorem}{Theorem}{Theorems}
\crefname{lemma}{Lemma}{Lemmas}
\crefname{remark}{Remark}{Remarks}
\crefname{problem}{Problem}{Problems}

\newtheorem{theorem}{Theorem}[section]
\newtheorem{example}{Example}[section]
\newtheorem{problem}{Problem}[section]
\newtheorem{remark}{Remark}[section]
\newtheorem{lemma}{Lemma}[section]
\newtheorem{definition}{Definition}[section]

\graphicspath{{Figures/}}

\begin{document}
	
\title{A direct imaging method for the exterior and interior inverse scattering  problems}

\author{
	Deyue Zhang\thanks{School of Mathematics, Jilin University, Changchun, China, {\it dyzhang@jlu.edu.cn}},
	Yue Wu\thanks{School of Mathematics, Jilin University, Changchun, China, {\it wuy20@mails.jlu.edu.cn}}, 
    Yinglin Wang\thanks{School of Mathematics, Jilin University, Changchun, China, {\it yinglin19@mails.jlu.edu.cn}}
    \ and 
	Yukun Guo\thanks{School of Mathematics, Harbin Institute of Technology, Harbin, China. {\it ykguo@hit.edu.cn} (Corresponding author)}
}
\date{}

\maketitle

\begin{abstract}
	This paper is concerned with the inverse acoustic scattering problems by an obstacle or a cavity with a sound-soft or a sound-hard boundary. A direct imaging method relying on the boundary conditions is proposed for reconstructing the shape of the obstacle or cavity. First, the scattered fields are approximated by the Fourier-Bessel functions with the measurements on a closed curve. Then, the indicator functions are established by the superpositions of the total fields or their derivatives to the incident point sources. We prove that the indicator functions vanish only on the boundary of the obstacle or cavity.  Numerical examples are also included to demonstrate the effectiveness of the method.
\end{abstract}



\section{Introduction}

The inverse scattering problems are of significant importance in diverse applications such as radar sensing, sonar detection and biomedical imaging (see, e.g. \cite{Colton}). In archetypal inverse scattering problems, the target objects are illuminated by an incident wave coming from their exterior and the corresponding scattering data is measured from the outside as well. Determining the unknown scatterer from such externally accessed information constitutes the exterior inverse scattering problems. In the last three decades, a huge number of computational attempts have been made to solve the exterior inverse scattering problems of identifying impenetrable obstacles or penetrable medium. Typical numerical strategies developed for the exterior inverse scattering problems include the decomposition methods, iteration schemes, recursive linearization based algorithms and the sampling approaches (see, e.g. \cite{BLLT15, Cakoni1, CK18, Colton}). We also refer to \cite{BL20, CDLZ20, DCL21} for some recent studies on the unique recovery issues in inverse scattering theory. In addition, inverse scattering problems without the phase information receive great interests recently. Some uniqueness results and numerical methods on the exterior inverse scattering problems with phaseless data can be found in \cite{BZ16, CH17, DZG19, JLZ19b, KR17, LL15, LLW17, ZZ17, ZG18a, ZG21}.

As the interior counterparts of the aforementioned exterior problems, the interior inverse scattering problems for recovering the shape of cavities rely on the signals due to interior emitters and sensors, which arise in many practical areas of non-destructive testing and reservoir exploration \cite{Jakubik, QC12a}. In contrast to the exterior inverse scattering problems, the interior inverse scattering problems are usually more complicated to tackle because of the repeated reflections of the 
inescapably trapped scattering waves. For the numerical reconstruction algorithms concerning interior inverse scattering problems, we refer to \cite{Hu, Liu14, Qin0, QC12a, QC12b, SGM16, Zeng1,  Zeng} for the linear sampling method, the regularized Newton iterative method, the decomposition method, the factorization method and the reciprocity gap functional method. There have also been some mathematical study on the inverse cavity problems. We refer to \cite{Hu, Liu14, QC12a, QC12b, Qin, Zeng} for some uniqueness results with full data (both the intensity and phase). A recent result on uniqueness of the inverse cavity scattering with phaseless data was established in \cite{ZWGL20} by utilizing the reference ball technique in conjunction with the superposition of incident point sources. 

In this paper, we consider the incident point sources and deal with the reconstruction of the shape of a sound-soft (or sound-hard) obstacle (or cavity). We propose a novel imaging scheme to determine the shape of unknown scatterer with known boundary conditions. Motivated by the Fourier-Bessel method for solving the Cauchy problems for the Helmholtz equation \cite{LZZL17, ZSMG20, ZS16}, we first make the approximation of the scattered fields by the expansion of Fourier-Bessel functions from the measurements on a closed curve. Then, by utilizing the a priori boundary conditions of the obstacle or the cavity, we introduce the associated indicator functions with the superposition of the approximated total fields or their derivatives with respect to the incident point sources. It is proved that the indicator functions vanish only on the boundary of the obstacle or cavity. Therefore, in the last step, profile of the underlying obstacle/cavity can be recognized as the zeros of the indicator functions depicted over a suitably chosen imaging domain. Mathematical analysis of the stability is presented to justify the theoretical foundations of the method.  Numerical examples are also included to demonstrate the effectiveness of the method.

In our opinion, the interesting novelties of the proposed method can be highlighted as follows. First, the proposed method can be viewed as a direct imaging method since the forward solver or the iteration process is not needed. Hence it is convenient to implement without time-consuming computations.  Second, the unified framework of inversion scheme can be well applied to both the exterior and interior problems with sound-soft or sound-hard boundary conditions, whereas the most existing methods for exterior problems produce drastically low-quality reconstructions for imaging the cavity. Third, rigorous mathematical justifications are provided to characterize the approximation properties with respect to the noisy data, which essentially guarantee the robustness of the algorithm. Finally, to our best knowledge, this is the first attempt in the literature towards combining the idea of direct imaging and the technique of Fourier-Bessel expansion for solving inverse scattering problems.

Throughout this paper, we assume that $D\subset\mathbb{R}^2$ is an open and simply connected domain with $C^2$ boundary $\partial D$.  Let $k>0$ be the wave number and $H_n^{(1)}$ be the Hankel function of the first kind of order $n$. For a generic point $z\in\mathbb{R}^2$, the incident field $u^i$ due to the point source located at $z$ is given by the fundamental solution to the Helmholtz equation
$$
	u^i (x; z):=\frac{\mathrm{i}}{4}H_0^{(1)}(k|x-z|)
$$
where
$$
\begin{cases}
	z\in \mathbb{R}^2\backslash\overline{D},\  x\in\mathbb{R}^2\backslash\overline{D}\cup\{z\}, & \text{for exterior problem},\\
	z\in D,\ x\in D\backslash\{z\}, & \text{for interior problem}.
\end{cases}
$$

The total field $u$ is the superposition of the incident field $u^i$ and the scattered field $u^s$, namely, $u=u^i+u^s$. We shall also employ $u(x; z)=u^i(x; z)+u^s(x; z)$ to indicate the dependence of these wave fields on the point source location $z$ and the wave number $k$. To characterize the physical properties of distinct scatterers, the boundary operator $\mathscr{B}$ is introduced by
\begin{align}\label{BC}
	\mathscr{B}u=
	\begin{cases}
		u, & \text{for a sound-soft obstacle/cavity},  \\
		\dfrac{\partial u}{\partial \nu}, & \text{for a sound-hard obstacle/cavity},
	\end{cases}
\end{align}
where $\nu$ is the unit outward normal to $\partial D$. 

The rest of this paper is arranged as follows. In the next section, we present the imaging method for the inverse obstacle scattering problem, including the model problem, uniqueness, indicator function as well as the Fourier-Bessel approximation. Then the interior counterparts for the inverse cavity scattering problem are discussed in \cref{sec:cavity}. Next, numerical validations and discussions of the proposed method are illustrated in \cref{sec:example}. Finally, some concluding remarks are given in section \cref{sec:conclusion}.

\section{Inverse obstacle scattering problem}\label{sec:obstacle}

We begin this section with the mathematical formulations of the model exterior scattering problem. The obstacle scattering problem can be formulated as: find the scattered field $u^s\in H^1_{\rm loc}(\mathbb{R}^2\backslash\overline{D})$ satisfying the following boundary value problem:
\begin{align}
    \Delta u^s+ k^2 u^s= &\, 0\quad \text{in}\ \mathbb{R}^2\backslash\overline{D},\label{Oeq:Helmholtz} \\
	\mathscr{B}(u^i+u^s)= &\, 0 \quad \text{on}\ \partial D, \label{OBC} \\
	\frac{\partial u^s}{\partial r}-\mathrm{i} ku^s= &\, o\left(r^{-1/2}\right), \quad r=|x| \to \infty, \label{SRC}
\end{align}
where the Sommerfeld radiation condition \cref{SRC} holds uniformly in all directions $x/|x|$. The existence of a solution to the direct scattering problem \cref{Oeq:Helmholtz,OBC,SRC} is well known (see, e.g., \cite{Colton}).

With these preparations,  the inverse obstacle scattering problem can be stated as the following.
\begin{problem}[Inverse obstacle scattering]\label{prob:obstacle}
	Let $D\subset B_\rho=\{x\in \mathbb{R}^2: |x|<\rho \}$ be the impenetrable obstacle with boundary condition $\mathscr{B}$ and $\Gamma\subset \mathbb{R}^2\backslash\overline{D}$  be a curve. Given the near-field data
	$$
		\{u(x; z):  x\in\partial B_\rho,\ z\in \Gamma\},
	$$
	for a fixed wave number $k$, determine the boundary $\partial D$ of the obstacle.
\end{problem}

We refer to \cref{fig:model_obstacle} for an illustration of the geometry setting of Problem \ref{prob:obstacle}. 
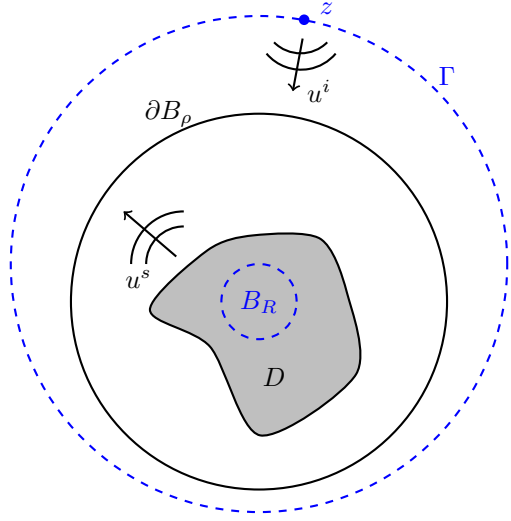
\begin{figure}
	\centering
	\begin{tikzpicture}[thick]
		\pgfmathsetseed{2021}
		\draw plot [smooth cycle, samples=8, domain={1:8}, xshift=1cm, yshift=0.5cm] (\x*360/8+5*rnd:0.8cm+1cm*rnd) [fill=lightgray] node at (1.2, -0.5) {$D$};
		\draw  (1, 0.5) circle (2.5cm) node at (-0.2, 3) {$\partial B_\rho$}; 
		\draw [dashed, blue] (1, 0.5) circle (0.5cm) node at (1, 0.5) {$B_R$}; 
		\draw [->] (-0.1, 1.1)--(-0.8, 1.7) node at (-0.6, 0.8) {$u^s$};
		\draw (0, 1.5) arc(90:180:0.5cm);
		\draw (0, 1.7) arc(90:180:0.7cm);
		\fill [blue] (1.6, 4.25) circle (2pt) node at (1.9, 4.4) {$z$};  
		\draw (1.2, 3.95) arc(225:315:0.5cm);
		\draw (1.1, 3.8) arc(220:320:0.6cm);
		\draw [->] (1.58, 4)--(1.45, 3.3) node at (1.8, 3.3) {$u^i$};
		\draw [blue, dashed] (1, 1) circle (3.3cm) node at (3.5, 3.5) {$\Gamma$}; 
	\end{tikzpicture}
	\caption{An illustration of the inverse obstacle scattering problem.} \label{fig:model_obstacle}
\end{figure}

\subsection{Uniqueness and the indicator functions}

From Theorem 2.1 in \cite{Colton}, we know that  the near-field data $\{u(x; z):  x\in\partial B_\rho,\ z\in \Gamma\}$ can uniquely determine the boundary $\partial D$ of the obstacle. The following theorem plays an important role in our numerical algorithm.
\begin{theorem}\label{Thm2.1}
	Let $D$ be a sound-soft obstacle. Assume that $\Lambda$ is the boundary of domain $D'$ and the total fields satisfy that
	\begin{equation}
		\int_{\Gamma}\int_{\Lambda}|u(x; z)| \mathrm{d}s(x)\mathrm{d}s(z)=0  \label{obstacle_condition1}.
	\end{equation}
	Then we have  $D'=D$ and $\Lambda=\partial D$.
\end{theorem}
\begin{proof}
	Assume that $D'\neq  D$ and $D_0=D\cap D'$. Without loss of generality, we assume that $ D^*=D\backslash\overline{D_0}$ is nonempty. From \cref{obstacle_condition1}, it is readily to see that $u(x; z)=0$ for all $ x\in \Lambda, z\in \Gamma$. Then, we see that for every $z\in \Gamma$,
	\begin{align*}
		\Delta u(\cdot; z)+k^2 u(\cdot; z) & = 0 \quad \text{in}\  D^*,\\
		u(\cdot; z) & = 0 \quad \text{on} \ \partial D^*.
	\end{align*}
	This implies that $u(\cdot; z)$ is a Dirichlet eigenfunction for the negative Laplacian in the domain $D^*$ with eigenvalue $k^2$. Further, from Theorem 2.1 in \cite{QC12a}, we know that the functions $u(\cdot; z)$ are linearly independent for distinct $z\in \Gamma$, which leads to a contraction since there are only finitely many linearly independent Dirichlet eigenfunctions $u_n$ (see \cite[ Theorem 5.1]{Colton}). Therefore, $D'=D$ and $\Lambda=\partial D$.
\end{proof}

\begin{theorem}\label{Thm2.2}
	Let $D$ be a sound-hard obstacle. Assume that $\Lambda$ is the boundary of domain $D'$ and the total fields satisfy that
    \begin{equation}
	\int_{\Gamma}\int_{\Lambda}\left|\frac{\partial u(x; z)}{\partial \nu(x)}\right| \mathrm{d}s(x)\mathrm{d}s(z)=0.  \label{obstacle_condition2}
    \end{equation}
	Then we have  $D'=D$ and $\Lambda=\partial D$.
\end{theorem}
\begin{proof}
 By \cref{obstacle_condition2}, we have $\frac{\partial u(x; z)}{\partial \nu(x)}=0$ for all $ x\in \Lambda, z\in \Gamma$, and that $D$ and $D'$ produce the same scattered fields, and then the far-field patterns coincide. Further, the mixed reciprocity relation \cite[Theorem 3.24]{Colton} implies that for all $z\in \Gamma$ and $\hat{x}\in \mathbb{S}^1=\{x\in \mathbb{R}^2: |x|=1\}$,
 \begin{align*}
 	v_D^s(z; -\hat{x})=v_{D'}^s(z; -\hat{x}),
 \end{align*}
 where $v_D^s(x;d)$ and $v_{D'}^s(x; d)$ are the scattered fields generated by the obstacle $D$ and $D'$ according to
 the incident plane wave $v^i(x; d)=\mathrm{e}^{\mathrm{i}k x\cdot d}$. This means the far field patterns coincide. And thus from  Theorem 5.6 in \cite{Colton}, we have $D=D'$ and $\Lambda=\partial D$.
\end{proof}

Now, we introduce an indicator function for the sound-soft obstacle
\begin{equation}\label{obstacleindicator1}
	I_{\rm s}(x)=\int_{\Gamma}|u(x; z)| \mathrm{d}s(z).
\end{equation}
From \cref{Thm2.1}, we only need to find a closed curve $\Lambda$ such that $I_{\rm s}=0$ on $\Lambda$.

For a sound-hard obstacle, we take a fixed $z_0\in\Gamma$, and let $\nu(x; z_0)$ be a unit vector such that $\nu(x; z_0) \cdot \nabla u(x; z_0)=0$. Then, we have $\nabla u(x; z)\cdot \nu(x; z_0)=0$ on $\partial D$ for every $z\in \Gamma$. Therefore, the indicator function for the sound-hard obstacle is as follows
\begin{equation}\label{obstacleindicator2}
	I_{\rm h}(x)  =  \int_{\Gamma}|\nabla u(x; z)\cdot \nu(x; z_0)| \mathrm{d}s(z). 
\end{equation}
 By \cref{Thm2.2}, one needs to find a closed curve $\Lambda$ such that $I_{\rm h}=0$ on $\Lambda$.

The approximation on $u(x; z)$ with the near-field data $u(x; z)|_{\partial B_\rho} $ will be given in the next subsection.

\subsection{The Fourier-Bessel approximation}

In this subsection we consider the approximation of the scattered fields $u^s(x; z)$ for every $z$ by the expansion of Fourier-Hankel functions
\begin{equation}
	u^s_N (x; z)= \sum_{n=-N}^N \hat{u}_n(z) \frac{H_n^{(1)}(kr)}{H_n^{(1)}(kR)} \mathrm{e}^{\mathrm{i} n\theta}, \quad r>R, \label{Fourier-Hankel}
\end{equation} 
with the Fourier coefficients
$$
\hat{u}_n(z)=\frac{1}{2\pi}\int_0^{2\pi} u^s(R,\theta; z) \mathrm{e}^{-\mathrm{i} n\theta}\, \mathrm{d}\theta,
$$
where the polar coordinates $(r, \theta): x=r(\cos\theta,\sin\theta)$ is used and $B_R=\{x\in \mathbb{R}^2: |x|<R \}\subset D$. Here similar to Chapter 5 in \cite{Colton} we make the assumption that the scattered fields can be extended analytically to $D\backslash B_R$.

By the orthogonality of the Fourier basis functions, the coefficients $\hat{u}_n(z)$ can be given by
\begin{equation}\label{coefficient}
	\hat{u}_n(z)=\frac{H_n^{(1)}(kR)}{2\pi H_n^{(1)}(k\rho)} \int_0^{2\pi} u^s(\rho,\theta; z)\mathrm{e}^{-\mathrm{i} n\theta}\, \mathrm{d}\theta,
\end{equation}
where $u^s(\cdot; z)|_{\partial B_\rho}=u(\cdot; z)|_{\partial B_\rho}-u^i(\cdot; z)|_{\partial B_\rho}$ is the  measured scattered data.

Taking the ill-posedness of underlying inverse problem into account, we have the following expansion of Fourier-Hankel functions
\begin{equation}\label{Noisy-Fourier-Hankel}
	u_N^{s,\delta} (x; z)=\sum_{n=-N}^N  \hat{u}_n^\delta(z) \frac{H_n^{(1)}(kr)}{ H_n^{(1)}(k\rho)}
	\mathrm{e}^{\mathrm{i} n\theta}, \quad R<r<\rho,
\end{equation}
where 
$$
	\hat{u}_n^\delta(z)=\frac{1}{2\pi} \int_0^{2\pi} u^{s,\delta}(\rho,\theta; z)\mathrm{e}^{-\mathrm{i} n\theta}\, \mathrm{d}\theta,
$$
and $u^{s,\delta}\in L^2(\partial B_\rho)$ are measured noisy data satisfying $\|u^{s,\delta}-u^{s}\|_{L^2(\partial B_\rho)}\leq \delta \|u^{s}\|_{L^2(\partial B_\rho)}$ with $0<\delta<1$. For simplicity,  we shall also use $\langle\cdot,\cdot\rangle_{\Lambda}$ for the usual inner product on $L^2(\Lambda)$ with $\Lambda$ being either $\partial B_\rho$ or $\partial B_R$ accordingly.

In what follows, $\Gamma(z)=\int_0^{\infty}\mathrm{e}^{-t}t^{z-1}\mathrm{d}t$ is the Gamma function, $J_n$ and $Y_n$ denote respectively the Bessel functions and Neumann functions of order $n$, and $[\,\cdot\,] $ signifies the rounding function which truncates the variable down to the nearest integer. To consider the error estimate, we need the following results.

\begin{lemma}\label{Lemma3.1}
	For $z>0$ and $n\in\mathbb{N}$ such that $n>(\mathrm{e}z+1)/2$, we have
    \begin{equation}\label{estimate1}
    	\frac{1}{2}\leq \frac{\pi z^n\left|H_n^{(1)}(z)\right|}{3\cdot 2^{n-1}\Gamma(n)}\leq \mathrm{e}^z.
    \end{equation}
\end{lemma}

\begin{proof}
	By the integral representations 
	$$
	J_n(z)=\frac{2(\frac{z}{2})^n}{\sqrt{\pi}\Gamma(n+1/2)}\int_0^1(1-t^2)^{n-1/2}\cos(zt) \mathrm{d}t,\quad n\geq 0,
	$$
	and
	$$
	Y_n(z)=\frac{2(\frac{z}{2})^n}{\sqrt{\pi}\Gamma(n+1/2)}
	\left(\int_0^1(1-t^2)^{n-1/2}\cos(zt) \mathrm{d}t-\int_0^\infty\mathrm{e}^{-zt}(1+t^2)^{n-1/2} \mathrm{d}t\right),\quad n\geq 0,
	$$
    and the fact that for $z>0$ and $n\geq 1$,
    \begin{align}
    	\int_0^1(1-t^2)^{n-1/2}\cos(zt) \mathrm{d}t & \le \int_0^1(1-t^2)^{n-1/2}\, \mathrm{d}t\notag\\
    	& = \int_0^{\pi/2} \cos^{2n}t\,\mathrm{d}t \notag\\
    	& = \frac{\pi}{2} \frac{(2n-1)!!}{(2n)!!}\notag\\
    	& = \frac{\pi}{2} \left(\frac{(2n-1)!!}{(2n)!!}\frac{(2n-1)!!}{(2n)!!}\right)^{1/2}\notag\\
    	& \le \frac{\pi}{2} \left(\frac{(2n-1)!!}{(2n)!!}\frac{(2n)!!}{(2n+1)!!}\right)^{1/2} \notag\\
    	& = \frac{\pi}{2\sqrt{2n+1}},\label{eq:wallis}
    \end{align}
    it is readily to see that for $z>0$ and $n\geq 1$,
	\begin{equation}\label{eqn1}
		|J_n(z)|\leq  \frac{2(\frac{z}{2})^n}{\sqrt{\pi}\Gamma(n+1/2)}\frac{\pi}{2\sqrt{2n+1}}
	\end{equation}
	and
	\begin{equation}\label{eqn2}
		|Y_n(z)|\leq\frac{2(\frac{z}{2})^n}{\sqrt{\pi}\Gamma(n+1/2)}\left(\frac{\pi}{2\sqrt{2n+1}}+\int_0^\infty\mathrm{e}^{-zt}(1+t^2)^{n-1/2} \mathrm{d}t\right).
	\end{equation}
	Moreover, by the Legendre duplication formula $\sqrt{\pi}\Gamma(2z)=2^{2z-1}\Gamma(z)\Gamma(z+1/2)$, we deduce that
	\begin{align}
		\int_0^\infty\mathrm{e}^{-zt}(1+t^2)^{n-1/2} \mathrm{d}t & \ge \int_0^\infty\mathrm{e}^{-zt}t^{2n-1} \mathrm{d}t \notag\\
		& = \frac{1}{z^{2n}}\int_0^\infty\mathrm{e}^{-t}t^{2n-1} \mathrm{d}t \notag \\ 
		& =  \frac{\Gamma(2n)}{z^{2n}} \notag\\
		& =  \frac{2^{2n-1}\Gamma(n)\Gamma(n+1/2)}{\sqrt{\pi}z^{2n}}, \label{eqn3}
	\end{align}
	and
   \begin{align}
	\int_0^\infty\mathrm{e}^{-zt}(1+t^2)^{n-1/2} \mathrm{d}t &\le \int_0^\infty\mathrm{e}^{-zt}(1+t)^{2n-1}\mathrm{d}t \notag \\
	& =  \mathrm{e}^z\int_0^\infty\mathrm{e}^{-z(t+1)}(1+t)^{2n-1}\mathrm{d}(t+1)\notag \\ 
	& \le \mathrm{e}^z\int_0^\infty\mathrm{e}^{-zt}t^{2n-1}\mathrm{d}t \notag \\
	& = \mathrm{e}^z\frac{\Gamma(2n)}{z^{2n}}. \label{eqn4}
	\end{align}
    Then from Stirling's approximation, 
	$$
	   n!>\sqrt{2\pi n}\left(\frac{n}{\mathrm{e}}\right)^n,
	$$
	we know that for $n>N_z:=(\mathrm{e}z+1)/2$,
	\begin{align}
		\frac{\Gamma(2n)}{z^{2n}} & =\frac{(2n-1)!}{z^{2n}}\notag \\
		& \ge \frac{\sqrt{2\pi (2n-1)}}{z}\left(\frac{2n-1}{\mathrm{e}z}\right)^{2n-1}\notag \\
		& \ge \frac{\mathrm{e}\sqrt{2\pi (2n-1)}}{2n-1} \notag \\
		& \ge \mathrm{e}\sqrt{\frac{2\pi}{2n+1}}\notag \\
		& \ge \frac{2\pi}{\sqrt{2n+1}}, \label{eqn5}
	\end{align}
    Now, by \cref{eqn1}, \cref{eqn2}, \cref{eqn4} and \cref{eqn5}, we deduce that for $n>N_z$,
	\begin{align}
		\left|H_n^{(1)}(z)\right| & =\left|J_n(z)+\mathrm{i}Y_n(z)\right|\notag\\
		& \le |J_n(z)|+|Y_n(z)| \notag\\
		& \le \frac{2(\frac{z}{2})^n}{\sqrt{\pi}\Gamma(n+1/2)}\left(\frac{\pi}{\sqrt{2n+1}}+\int_0^\infty\mathrm{e}^{-zt}(1+t^2)^{n-1/2} \mathrm{d}t\right)\notag\\
		& \le \frac{z^n}{\sqrt{\pi}2^{n-1}\Gamma(n+1/2)}\left(\mathrm{e}^z+\frac{1}{2}\right)\frac{\Gamma(2n)}{z^{2n}}\notag\\
		& \le \frac{3\mathrm{e}^z}{2}\frac{z^n}{\sqrt{\pi}2^{n-1}\Gamma(n+1/2)}\frac{2^{2n-1}\Gamma(n)\Gamma(n+\frac{1}{2})}{\sqrt{\pi}z^{2n}}\notag\\
		& = \mathrm{e}^z\frac{3\cdot 2^{n-1}\Gamma(n)}{\pi z^n}.\label{eq:Hn_upper_estimate}
	\end{align}

	Similarly, from the integral representation of $Y_n(z)$, \cref{eq:wallis}, \cref{eqn3} and \cref{eqn5}, it can be seen that for $n>N_z$,
	\begin{align}
		\left| H_n^{(1)}(z)\right| & \ge  \left|Y_n(z)\right|\notag\\
		& \ge \frac{2(\frac{z}{2})^n}{\sqrt{\pi}\Gamma(n+1/2)}
		\left(\int_0^\infty\mathrm{e}^{-zt}(1+t^2)^{n-1/2}\mathrm{d}t-\frac{\pi}{2\sqrt{2n+1}}\right)\notag\\
		& \ge \frac{2(\frac{z}{2})^n}{\sqrt{\pi}\Gamma(n+1/2)} \left(\frac{\Gamma(2n)}{z^{2n}}-\frac{\pi}{2\sqrt{2n+1}}\right)\notag\\
		& \ge \frac{3}{4} \frac{z^n}{\sqrt{\pi}2^{n-1}\Gamma(n+1/2)}\frac{\Gamma(2n)}{z^{2n}}\notag\\
		& = \frac{1}{2} \frac{3\cdot 2^{n-1}\Gamma(n)}{\pi z^n}. \label{eq:Hn_lower_estimate}
	\end{align}

	Finally, combining \cref{eq:Hn_upper_estimate} and \cref{eq:Hn_lower_estimate}, we arrive at estimate \cref{estimate1} and this completes the proof.
\end{proof}

\begin{lemma}[{\cite[Lemma 3.3]{ZG15}}]\label{lem:hankel_estimate}
	Let $n\in \mathbb{Z}$, and $t, t_1, t_2 \in \mathbb{R}_+$ satisfy $t_1\leq t_2$. Then the following estimates hold
	\begin{equation}\label{Lemma3.2.1}
		|H^{(1)}_n(t_2)| \le |H^{(1)}_n(t_1)|.
	\end{equation}
\end{lemma}

From \cref{Lemma3.1,lem:hankel_estimate}, we obtain the following approximation result.
\begin{theorem}\label{thm:estimates_obstacle}
	Let $\gamma=\mathrm{dist}(\partial B_R, \partial D)$ and $N\geq N_0:=\frac{\mathrm{e} k\rho+1}{2}.$
	Then there exist positive constants $C_1, C_2, \cdots, C_6$, such that
	\begin{align}
		\left \|u^s-u_N^{s,\delta}\right\|_{L^2(\partial B_\rho)} & \le C_1 \tau_1^{-N}+C_2 N^{1/2}\delta,\label{estimate3} \\ 
		\left \|u^s-u_N^{s,\delta}\right\|_{L^2(\partial D)} & \le C_3\tau_2^{-N}+C_4\tau_3^{N}\delta , \label{estimate4} \\
		\left \|\partial_{\nu}u^s-\partial_{\nu}u_N^{s,\delta}\right\|_{L^2(\partial D)} & \le C_5N\tau_2^{-N}+C_6N\tau_3^{N}\delta,  \label{estimate41}
	\end{align}
	where $\tau_1=\frac{\rho}{R}, \tau_2={\frac{R+\gamma}{R}}$,$\tau_3=\frac{\rho}{R+\gamma}$. Moreover, for $0<\delta<\tau_1^{-N_0}$, take $N= \left[\frac{1}{\ln \tau_1}\ln \frac{1}{\delta}\right]$, then the following results hold
	\begin{align}
		\left \|u^s-u_N^{s,\delta}\right\|_{L^2(\partial B_\rho)} & \le C_1 \delta+\frac{C_2}{\sqrt{\ln \tau_1}}   \delta^{1/2}, \label{estimate5} \\ 
		\left \|u^s-u_N^{s,\delta}\right\|_{L^2(\partial D)} & \le (C_3+C_4)\delta^{\frac{\ln \tau_2}{\ln \tau_1}}, \label{estimate6} \\
		 \left \|\partial_{\nu}u^s-\partial_{\nu}u_N^{s,\delta}\right\|_{L^2(\partial D)} & \le \frac{C_5+C_6}{\ln \tau_1}\delta^{\frac{\ln \tau_2}{\ln \tau_1}}|\ln \delta|.  \label{estimate61}
	\end{align}
\end{theorem}

\begin{proof}
	From \cref{Fourier-Hankel} and \cref{estimate1}, it can be seen that
	\begin{align*}
		\left \|u^s-u^s_N\right\|^2_{L^2(\partial B_\rho)} & =  2\pi\rho  \sum_{|n|>N} \left|\frac{H_n^{(1)}(k\rho)}{H_n^{(1)}(kR)}\right|^2 |\hat{u}_n(z)|^2\\
		& \le  2\pi\rho  \sum_{|n|>N}4\mathrm{e}^{2k\rho} \left(\frac{R}{\rho}\right)^{2|n|}|\hat{u}_n(z)|^2\\
		& \leq 8\pi\rho \mathrm{e}^{2k\rho} \left(\frac{R}{\rho}\right)^{2N}\|u^s\|_{L^2(\partial B_R)}^2.
	\end{align*}
	Further, by \cref{Fourier-Hankel}-\cref{Noisy-Fourier-Hankel}, we  have
	\begin{align*}
		\left \|u^s_N-u_N^{s,\delta}\right\|_{L^2(\partial B_\rho)}^2
		& = \frac{1}{2\pi\rho}\sum_{n=-N}^N  \left |\left\langle u^s-u^{s,\delta}, \mathrm{e}^{\mathrm{i} n\theta}\right\rangle_{\partial B_\rho} \right|^2\\
		& \le (2N+1)\delta^2 \|u^{s}\|_{L^2(\partial B_\rho)}^2.
	\end{align*}
	And thus,
	\begin{align*}
		\left \|u^s-u_N^{s,\delta}\right\|_{L^2(\partial B_\rho)} & \le \left \|u^s-u^s_N\right\|_{L^2(\partial B_\rho)}+ \left \|u^s_N-u_N^{s,\delta}\right\|_{L^2(\partial B_\rho)}\\
		& \le C_1 \tau_1^{-N}+C_2 N^{1/2}\delta,
	\end{align*}
	where $\tau_1=\frac{\rho}{R}$, $C_1=2\sqrt{2\pi\rho}\mathrm{e}^{k\rho}\|u^s\|_{L^2(\partial B_R)}$ and
	$ C_2=\sqrt{3}\|u^{s}\|_{L^2(\partial B_\rho)}$. This leads to estimate \cref{estimate3}.
	
	From \cref{estimate1} and \cref{Lemma3.2.1}, it is obvious that for $x\in \partial D$,
	\begin{align*}
		\left |u^s(x; z)-u^s_N(x; z)\right| & \le  \sum_{|n|>N} |\hat{u}_n(z)| \left |\frac{H_n^{(1)}(kr)}{H_n^{(1)}(kR)} \right|\\
		& \le \sum_{|n|>N} |\hat{u}_n(z)| \left |\frac{H_n^{(1)}(k(R+\gamma))}{H_n^{(1)}(kR)} \right|\\
		& \le   \sum_{|n|>N}2\mathrm{e}^{k(R+\gamma)} \left(\frac{R}{R+\gamma}\right)^{|n|}|\hat{u}_n(z)|  \\
		& \le  2\mathrm{e}^{k(R+\gamma)}\tau_2^{-N}\left(\sum_{|n|\geq1}\tau_2^{-2 |n|}\right)^{1/2} \left(\sum_{|n|>N}|\hat{u}_n(z)|^2\right)^{1/2}\\
		& \le \widetilde{C}_3\tau_2^{-N} \|u^s\|_{L^2(\partial B_R)},
	\end{align*}
	where $\gamma=\mathrm{dist}(\partial B_R, \partial D)$, $\tau_2= \frac{R+\gamma}{R}$   and $\widetilde{C}_3=2\sqrt{2}\mathrm{e}^{k(R+\gamma)}(\tau^2_2-1)^{-1/2} $. This implies
	\begin{equation}\label{eqn7}
		\left \|u^s-u^s_N\right\|_{L^2(\partial D)}\leq  C_3\tau_2^{-N},
	\end{equation}
	where $C_3=\widetilde{C}_3\|u^s\|_{L^2(\partial B_R)}|\partial D|^{1/2}$ and $|\partial D|$ denotes the length of $\partial D$. By using estimates \cref{estimate1} and \cref{Lemma3.2.1}, we have that for $x\in \partial D$,
	\begin{align*}
		\left |u^s_N(x; z)-u_N^{s,\delta}(x; z)\right|
		& \le \frac{1}{2\pi\rho}\sum_{n=-N}^N  \left |\frac{H_n^{(1)}(kr)}{H_n^{(1)}(k\rho)} \right|
		\left |\left\langle u^s-u^{s,\delta}, \mathrm{e}^{\mathrm{i} n\theta}\right\rangle_{\partial B_\rho} \right|\\
		& \le \frac{\|u^{s}\|_{L^2(\partial B_\rho)}}{\sqrt{2\pi\rho}}\delta\sum_{n=-N}^N \left |\frac{H_n^{(1)}(k(R+\gamma))}{H_n^{(1)}(k\rho)} \right|\\
		& \le \frac{\|u^{s}\|_{L^2(\partial B_\rho)}}{\sqrt{2\pi\rho}}\delta\sum_{n=-N}^N 2\mathrm{e}^{k(R+\gamma)}\left(\frac{\rho}{R+\gamma}\right)^{|n|}\\
		& \le \frac{5\|u^{s}\|_{L^2(\partial B_\rho)}}{3\sqrt{\pi\rho}}\delta\sum_{n=-N}^N \left(\frac{2\rho}{R+\gamma}\right)^{|n|}\\
		& \le \widetilde{C}_4\delta \tau_3^{N+1}\|u^{s}\|_{L^2(\partial B_\rho)},
	\end{align*}
	where $\tau_3=\frac{\rho}{R+\gamma}$ and $\widetilde{C}_4=\frac{ 2\sqrt{2}\mathrm{e}^{k(R+\gamma)}}{\sqrt{\pi\rho}(\tau_3-1)}$.  This, together with \cref{eqn7}, yields
	\begin{align*}
		\left \|u^s-u_N^{s,\delta}\right\|_{L^2(\partial D)}&\le \left \|u^s-u^s_N\right\|_{L^2(\partial D)}
		+ \left \|u^s_N-u_N^{s,\delta}\right\|_{L^2(\partial D)}\\
		&\le C_3\tau_2^{-N}+C_4\delta \tau_3^N,
	\end{align*}
	where $C_4=\widetilde{C}_4\tau_3|\partial D|^{1/2}$, which implies \cref{estimate4}.
	
	From the formula ${H_n^{(1)}}'(t)=-H_{n+1}^{(1)}(t)+n H_n^{(1)}(t)/t$, we have we have
	$$
	    \left|\nabla H_n^{(1)}(kr)\right|\le k\left|{H_{n+1}^{(1)}}(kr)\right|+\frac{|n|}{r}\left|{H_{n}^{(1)}}(k r)\right|,
	$$
	which, in conjunction with $\left|\nabla  \mathrm{e}^{\mathrm{i}n \theta}\right|=\frac{|n|}{r}$, implies
	\begin{equation}\label{gradient}
		   \left|\nabla \left(H_n^{(1)}(kr)\mathrm{e}^{\mathrm{i}n \theta}\right)\right|\leq k\left|{H_{n+1}^{(1)}}(kr)\right|+\frac{2|n|}{r}\left|{H_{n}^{(1)}}(k r)\right|.
	\end{equation}
	Now, by estimates \cref{estimate1}, \cref{Lemma3.2.1} and \cref{gradient}, we derive that for $x\in \partial D$,
	\begin{align*}
		 \left |\frac{\partial}{\partial \nu}u^s(x; z)-\frac{\partial}{\partial \nu}u_N^s(x; z)\right| 
		 & \le  \sum_{|n|>N} |\hat{u}_n(z)| \left |\frac{\nabla \left(H_n^{(1)}(kr)\mathrm{e}^{\mathrm{i}n \theta}\right)}{H_n^{(1)}(kR)} \right| \\
		 & \le \sum_{|n|>N} |\hat{u}_n(z)| \left(k\left |\frac{H_{n+1}^{(1)}(kr)}{H_n^{(1)}(kR)} \right|+\frac{2|n|}{r}\left |\frac{H_n^{(1)}(kr)}{H_n^{(1)}(kR)} \right|\right)\\
	     & \le \frac{8\mathrm{e}^{k(R+\gamma)}}{R+\gamma}\sum_{|n|>N}|n|\left(\frac{R}{R+\gamma}\right)^{|n|} |\hat{u}_n(z)|  \\
		& \le \widetilde{C}_5 N\tau_2^{-N} \|u^s\|_{L^2(\partial B_R)},
	\end{align*}
	where 
	$$
	\widetilde{C}_5=\frac{8\sqrt{2}\mathrm{e}^{k(R+\gamma)}}{R+\gamma}\left(\sum_{|n|\geq1}n^2\tau_2^{-2 |n|}\right)^{1/2}.
	$$
	This implies
	\begin{equation}\label{eq:estimate_C5}
		\left \|\frac{\partial}{\partial \nu}u^s(x; z) -\frac{\partial}{\partial \nu}u_N^s(x; z) \right\|_{L^2(\partial D)}\leq  C_5 N\tau_2^{-N},
	\end{equation}
	where $C_5=\widetilde{C}_5\|u^s\|_{L^2(\partial B_R)}|\partial D|^{1/2}$.
	Again, estimates \cref{estimate1}, \cref{Lemma3.2.1} and \cref{gradient}, we see that for $x\in \partial D$,
	\begin{align*}
		\left |\frac{\partial}{\partial \nu}u_N^s(x; z) -\frac{\partial}{\partial \nu}u_N^{s,\delta}(x; z)\right| 
		& \le \frac{1}{2\pi\rho}\sum_{n=-N}^N  \left |\frac{\nabla \left(H_n^{(1)}(kr)\mathrm{e}^{\mathrm{i}n \theta}\right)}{H_n^{(1)}(k\rho)} \right|
		\left|\left\langle u^s-u^{s,\delta}, \mathrm{e}^{\mathrm{i} n\theta}\right\rangle_{\partial B_\rho} \right|\\
		& \le\frac{1}{\sqrt{2\pi\rho}}\sum_{n=-N}^N\left(k\left |\frac{H_{n+1}^{(1)}(kr)}{H_n^{(1)}(kR)} \right|+\frac{2|n|}{r}\left |\frac{H_n^{(1)}(kr)}{H_n^{(1)}(kR)} \right|\right)
		\delta \|u^{s}\|_{L^2(\partial B_\rho)}\\
		& \le \frac{8\mathrm{e}^{k(R+\gamma)}\|u^{s}\|_{L^2(\partial B_\rho)}}{\sqrt{2\pi\rho}(R+\gamma)}\delta\sum_{n=-N}^N |n|\left(\frac{\rho}{R+\gamma}\right)^{|n|}\\
		& \le \widetilde{C}_6\delta N\tau_3^{N+1}\|u^{s}\|_{L^2(\partial B_\rho)},
	\end{align*}
	where  $\widetilde{C}_6=\frac{8\sqrt{2}\mathrm{e}^{k(R+\gamma)}}{\sqrt{\pi\rho}(R+\gamma)(\tau_3-1)}$. This, together with \cref{eq:estimate_C5}, yields
	\begin{align*}
		\left \|\partial_{\nu}u^s-\partial_{\nu}u_N^{s,\delta}\right\|_{L^2(\partial D)} & \le
		\left \|\partial_{\nu}u^s-\partial_{\nu}u_N^s\right\|_{L^2(\partial D)}
		+ \left \|\partial_{\nu}u_N^s-\partial_{\nu}u_N^{s,\delta}\right\|_{L^2(\partial D)}\\
		& \le C_5N\tau_2^{-N}+C_6\delta N\tau_3^N,
	\end{align*}
	where $C_6=\widetilde{C}_6\tau_3\|u^{s}\|_{L^2(\partial B_\rho)}|\partial D|^{1/2}$, which implies \cref{estimate41}.
	
	For $0<\delta<\tau_1^{-N_0}$, from $N=\left[\frac{1}{\ln \tau_1}\ln \frac{1}{\delta}\right]$, we see $N\geq N_0$. Substituting $N=\left[\frac{1}{\ln \tau_1}\ln \frac{1}{\delta}\right]$ into \cref{estimate3}, \cref{estimate4}, \cref{estimate41}, and using $\ln t<t$ for $t>1$, we derive the estimates \cref{estimate5}, \cref{estimate6} and \cref{estimate61}. The proof is complete.
\end{proof}

With the approximation of the scattered fields $u^s(x; z)$ by $u_N^{s, \delta} (x; z)$ in \cref{Noisy-Fourier-Hankel}, we now introduce the approximating indicator functions.

(i) The sound-soft case. The approximating indicator function for \cref{obstacleindicator1} is defined by
\begin{equation}\label{obstacleindicator4}
	\mathcal{I}_{{\rm s}, N}(x)=\int_{\Gamma}\left|u_N^{s, \delta} (x; z)+u^i(x; z)\right| \mathrm{d}s(z)
\end{equation}

(ii) The sound-hard case. The approximating indicator function for \cref{obstacleindicator2} is defined by
\begin{equation}\label{obstacleindicator5}
	\mathcal{I}_{{\rm h}, N}(x)=  \int_{\Gamma}\left|\nabla\left(u_N^{s, \delta} (x; z) +u^i(x; z)\right)\cdot \nu^{\delta}(x; z_0)\right| \mathrm{d}s(z) , 
\end{equation}
where $\nu^{\delta}(x; z_0)$ is a unit vector satisfying $\nu^{\delta}(x; z_0)\cdot \nabla(u_N^{s, \delta} (x; z_0) +u^i(x; z_0))=0$ and
\begin{equation}\label{unitvector}
	\left|\nabla\left(u_N^{s,\delta}(x; z_0)+u^i(x; z_0) \right)\right|=\max_{z\in \Gamma}\left\{\left|\nabla\left(u_N^{s,\delta}(x; z) +u^i(x; z)\right)\right|\right\}.
\end{equation}

Based on Theorem \ref{thm:estimates_obstacle}, the following estimates for the indicator functions hold.
\begin{theorem}\label{thm:estimates_indicator_obstacle}
	Let $\tau_1,\tau_2$ and $N_0$ be defined in \cref{thm:estimates_obstacle}. Then for $0<\delta<\tau_1^{-N_0}$ and $N=\left[\frac{1}{\ln \tau_1}\ln \frac{1}{\delta}\right]$, there exist positive constants $C_{\rm s}$ and $C_{\rm h}$, such that
	\begin{align}
		\mathcal{I}_{{\rm s}, N}(x) & \le C_{\rm s} \delta^{\alpha},\label{Thm3.2Est1} \\ 
		\mathcal{I}_{{\rm h}, N}(x) & \le C_{\rm h} \delta^{\alpha}|\ln \delta|, \label{Thm3.2Est2}
	\end{align}
	where $\alpha=\ln \tau_2/\ln \tau_1$.
\end{theorem}
\begin{proof}
	From the proof of \cref{estimate4} in \cref{thm:estimates_obstacle} and $\tau_3=\tau_1/\tau_2$, it can be seen that for $x\in \partial D$,
	$$
		\left |u^s(x; z)-u^{s,\delta}_N(x; z)\right|\leq  C_7\left(\tau_2^{-N} \|u^s\|_{L^2(\partial B_R)}+\delta \tau_3^N \|u^{s}\|_{L^2(\partial B_\rho)}\right),
	$$
	where $C_7>0$ is a constant. By taking $N= \left[\frac{1}{\ln \tau_1}\ln \frac{1}{\delta}\right]$, we know that for $x\in \partial D$,
	\begin{align*}
		\left |u^s(x; z)-u^{s,\delta}_N(x; z)\right| & \le C_7\left(\delta^{\frac{\ln \tau_2}{\ln \tau_1}} \|u^s\|_{L^2(\partial B_R)}+ \delta^{1-\frac{\ln \tau_3}{\ln \tau_1}}\|u^{s}\|_{L^2(\partial B_\rho)}\right)\\
		& = C_7\left( \|u^s\|_{L^2(\partial B_R)}+\|u^s\|_{L^2(\partial B_\rho)}\right)\delta^{\frac{\ln \tau_2}{\ln \tau_1}}.
	\end{align*}
	And thus, for $x\in \partial D$,
	\begin{align*}
		\mathcal{I}_{{\rm s}, N}(x) & = \int_{\Gamma}\left|u_N^{s,\delta} (x; z)-u^s(x; z) +u^s(x; z)+u^i(x; z)\right| \mathrm{d}s(z) \\
		& = \int_{\Gamma}\left|u_N^{s, \delta} (x; z)-u^s(x; z)\right| \mathrm{d}s(z)\\
		& \le C_{\rm s} \delta^{\alpha},
	\end{align*}
    where 
    $$
    \alpha=\frac{\ln \tau_2}{\ln \tau_1},  \quad C_{\rm s}= C_7 C_0,\quad C_0:=\int_{\Gamma}\left(\|u^s(\cdot; z) \|_{L^2(\partial B_R)}+\|u^s(\cdot; z)\|_{L^2(\partial B_\rho)}\right) \mathrm{d}s(z).
    $$ 
    This leads to the estimate \cref{Thm3.2Est1}.
    
    Similarly, following the proof of \cref{estimate41} in \cref{thm:estimates_obstacle}, we have that for $x\in \partial D$,
    $$
    	\left|\nabla\left(u^s(x; z)- u^{s, \delta}_N(x; z)\right)\right|\leq  C_8\left( \|u^s\|_{L^2(\partial B_R)}+ \|u^s\|_{L^2(\partial B_\rho)}\right)\delta^{\alpha}|\ln \delta|,
    $$  
    where $C_8>0$ is a constant. This means that for $x\in \partial D$,
    \begin{equation}\label{Thm3.2Est4}
    	\int_{\Gamma}\left|\nabla\left(u_N^{s, \delta} (x; z)-u^s(x; z)\right)\cdot \nu(x; z_0)\right| \mathrm{d}s(z)\leq  C_8 C_0\delta^{\alpha}|\ln \delta|.
    \end{equation}

    For $\eta=(\eta_1, \eta_2)^{\top}\in \mathbb{R}^2$, we introduce the notation $\eta^{\bot}=(-\eta_2,\eta_1)^{\top}\in \mathbb{R}^2$. For convenience, denote $\xi_\delta(x; z) :=\nabla\left(u_N^{s, \delta}(x; z)+u^i(x; z)\right)$ and $\xi(x; z):=\nabla\left(u^s(x; z)+u^i(x; z)\right)$. Since for $x\in \partial D$, $\nu^{\delta}(x; z_0)\cdot \nabla\left(u_N^{s,\delta}(x; z_0)+u^i(x; z_0)\right)=0$ and $\nu(x; z_0)\cdot \nabla\left(u^s (x; z_0)+u^i(x; z_0)\right)=0$, it is readily to see that
    $$
    	\nu^{\delta}(x; z_0)=\frac{\xi_\delta^{\bot}(x; z_0)}{|\xi_\delta^{\bot}(x; z_0)| }, \quad
    	\nu(x; z_0)=\frac{\xi^{\bot}(x; z_0)}{|\xi^{\bot}(x; z_0)|}
    $$
    and
    \begin{align*}
    	\left|\nu^{\delta}(x; z_0)-\nu(x; z_0)\right| & \le\frac{2\left|\xi_\delta^{\bot}(x; z_0)-\xi^{\bot}(x; z_0)\right|}
    	{\left|\xi_\delta^{\bot}(x; z_0)\right|}\\
    	& =\frac{2\left|\xi_\delta(x; z_0)-\xi(x; z_0)\right|}
    	{\left|\xi_\delta(x; z_0)\right|}\\
    	& =\frac{2\left|\nabla\left(u_N^{s, \delta} (x; z_0)-u^s(x; z_0)\right)\right|}
    	{\left| \nabla\left(u_N^{s,\delta}(x; z_0)+u^i(x; z_0)\right)\right|}.
    \end{align*}
    Further, from \cref{unitvector}, we see that for $x\in \partial D$,
    \begin{align*}
    	&\quad \int_{\Gamma}\left|\left(\nu^{\delta}(x; z_0)-\nu(x; z_0)\right)\cdot \nabla\left(u_N^{s,\delta}(x; z)+u^i(x; z)\right)\right|\mathrm{d}s(z) \\ 
    	& \le 2\int_{\Gamma}\frac{\left|\nabla\left(u_N^{s,\delta} (x; z_0)-u^s(x; z_0)\right)\right|} {\left|\nabla\left(u_N^{s, \delta}(x; z_0)+u^i(x; z_0)\right)\right|}\left|\nabla\left(u_N^{s, \delta}(x; z)+u^i(x; z)\right)\right|
    	\mathrm{d}s(z) \\
    	& \le 2 C_8 C_0\delta^{\alpha}|\ln \delta|,
    \end{align*}
    which, together with \cref{obstacleindicator5} and \cref{Thm3.2Est4}, yields that for $x\in \partial D$,
    \begin{align*} 
    	&\quad \int_{\Gamma}\left|\nabla\left(u_N^{s,\delta} (x; z)+u^i(x; z)\right)\cdot \nu^{\delta}(x; z_0)\right| \mathrm{d}s(z) \\
    	& \le \int_{\Gamma}\left|\nabla\left(u_N^{s,\delta} (x; z)+u^i(x; z)\right)\cdot \left(\nu^{\delta}(x; z_0)-\nu(x; z_0)\right)\right| \mathrm{d}s(z) \\
    	&\quad + \int_{\Gamma}\left|\nabla\left(u_N^{s,\delta} (x; z)-u^s(x; z)\right)\cdot \nu(x; z_0)\right| \mathrm{d}s(z)\\
    	& \le  3 C_8 C_0\delta^{\alpha}|\ln \delta|.
    \end{align*}
    This implies estimate \cref{Thm3.2Est2} holds.
\end{proof}

\section{Inverse cavity scattering problem}\label{sec:cavity}

We begin this section with the precise formulations of the model cavity scattering problem. The interior scattering problem for cavities can be formulated as: to find the scattered field $u^s\in H^1(D)$ which satisfies the following boundary value problem:

\begin{align}
	\Delta u^s+ k^2 u^s & = 0\quad \mathrm{in}\ D,\label{eq:Helmholtz} \\
	\mathscr{B}u & = 0 \quad \mathrm{on}\ \partial D, \label{eq:boundary_condition}
\end{align}
where $u=u^i+u^s$ denotes the total field. Provided that $k^2$ is not a Dirichlet eigenvalue for $-\Delta$ in $D$, it is well known that the direct scattering problem \cref{eq:Helmholtz}--\cref{eq:boundary_condition} admits a unique solution $u^s\in H^1(D)$  (see, e.g. \cite{Cakoni1, Colton}). From now on we assume that $k^2$ is not a Dirichlet eigenvalue for $-\Delta$ in $D$.

Following \cite{ZWGL20}, we introduce the following definition of admissible curve.
\begin{definition} (Admissible curve) An open curve $\Gamma$ is called an admissible curve with respect to domain $\Omega$ if

(i) $\Omega \subset  D$ is simply-connected; 

(ii) $ \partial \Omega$ is analytic homeomorphic to $\mathbb{S}^1$; 

(iii) $k^2$ is not a Dirichlet eigenvalue of $-\Delta$ in $\Omega$; 

(iv) $\Gamma \subset \partial \Omega$ is a one-dimensional (1D) analytic manifold with nonvanishing measure.
\end{definition} 

\begin{remark}
We would like to remark that this requirement for the admissibility of $\Gamma$ can be easily fulfilled. For example, since the first zero of $J_0$ is approximately $2.4048$, $\Omega$ can be chosen as a disk whose radius is less than $2.4048/k$ and $\Gamma$ is chosen as an arbitrary corresponding semicircle.
\end{remark}

Now, we introduce the interior inverse scattering problem for incident point sources. 
\begin{problem}[Inverse cavity scattering]\label{prob:cavity}
	Let $D$ be the impenetrable cavity with boundary condition $\mathscr{B}$. Assume that $\Gamma$ and $\Gamma_R$ are admissible  curves  with respect to $\Omega$ and $B_R$, respectively, such that  $\overline{\Omega}\subset B_R$ and $B_R\subset D$. Given the total field data
	\begin{equation*}
		\{u(x; z):  x\in \partial B_R,\ z\in \Gamma\},
	\end{equation*}
	for a fixed wavenumber $k$, determine the boundary $\partial D$ of the cavity.
\end{problem}

We refer to \cref{fig:model_cavity} for an illustration of the geometry setting of \cref{prob:cavity}. 

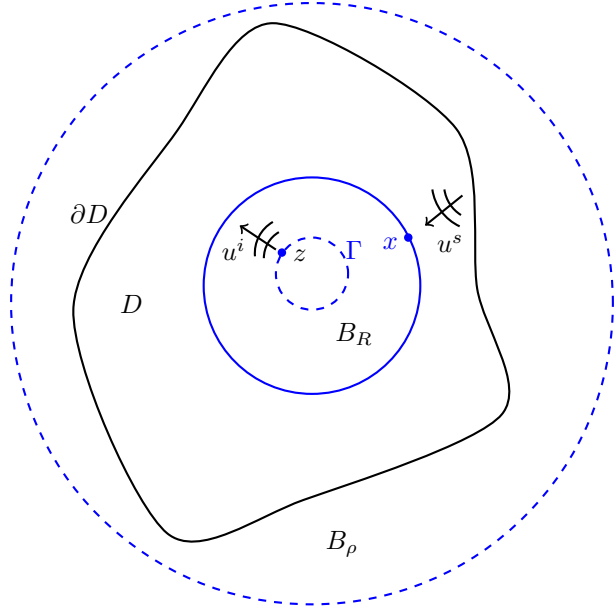
\begin{figure}
	\centering
	\begin{tikzpicture}[thick, scale=0.8]
		\pgfmathsetseed{2022}
		\draw plot [smooth cycle, samples=8, domain={1:8},  xshift=1.3cm, yshift=3.2cm] (\x*360/8+5*rnd:0.5cm+5.5cm*rnd) node at (-1.5, 3) {$D$};
		
		\draw node at (-2.2, 4.5) {$\partial D$};
		
		\draw [blue] (1.5, 3.3) circle (1.8cm);
		\draw node at (2.2, 2.5) {$B_R$};
		\draw [dashed, blue] (1.5, 3) circle (5cm); 
		\draw node at (2, -1) {$B_\rho$}; 
		
		\draw [blue, dashed] (1.5, 3.5) circle (0.6cm) node at (2.2, 3.9) {$\Gamma$};
		\draw node at (1.3, 3.8) {$z$};
		\fill [blue] (1, 3.85) circle (2pt);  
		\draw [->] (0.9, 3.9)--(0.3, 4.3) node at (0.2, 3.9) {$u^i$};;
		\draw (0.95,4.19) arc(120:180:0.5cm);
		\draw (0.85,4.36) arc(120:185:0.6cm);
		
		\draw [->](4, 4.8)-- (3.38,4.3);
		\fill [blue] (3.1, 4.1) circle (2pt) node at (2.8, 4) {$x$};  

		\draw (3.7, 4.9) arc(188:238:0.6cm);
		\draw (3.5, 4.9) arc(188:245:0.8cm);
		
		\draw node at (3.8, 4) {$u^s$}; 
	\end{tikzpicture}
	\caption{An illustration of the interior inverse scattering problem.} \label{fig:model_cavity}
\end{figure}

\subsection{Uniqueness and the indicator functions}

From \cite[Theorem 2.1]{QC12a} and \cite[Theorem 3.1]{QC12b}, we know that  the total field data $\{u(x; z):  x\in \partial B_R,\ z\in \Gamma\} $ can uniquely determine the boundary $\partial D$ of the cavity. Analogous to the inverse obstacle scattering problem, we have the following uniqueness results  and the indicator functions.

\begin{theorem}\label{Thm4.1}
	Let $D$ be a sound-soft cavity. Assume that $\Lambda$ is the boundary of domain $D'$. Suppose that the total fields satisfy that
	\begin{equation}
		\int_{\Gamma}\int_{\Lambda}|u(x; z)| \mathrm{d}s(x)\mathrm{d}s(z)=0  \label{cavity_condition1}.
	\end{equation}
	Then we have  $D'=D$ and $\Lambda=\partial D$.
\end{theorem}
\begin{proof}
 See \cref{Thm2.1} for a similar proof.
\end{proof}

\begin{theorem}\label{Thm4.2}
	Let $D$ be a sound-hard cavity. Assume that $\Lambda$ is the boundary of domain $D'$. Suppose that the total fields satisfy that
	\begin{equation}
		\int_{\Gamma}\int_{\Lambda}\left|\frac{\partial u(x; z)}{\partial \nu(x)}\right| \mathrm{d}s(x)\mathrm{d}s(z)=0  \label{cavity_condition2}.
	\end{equation}
	Then we have  $D'=D$ and $\Lambda=\partial D$.
\end{theorem}
\begin{proof}
	By \cref{cavity_condition2} and the definition of admissible curve $B_R$, we have $\frac{\partial u(x; z)}{\partial \nu(x)}=0$ for all $ x\in \Lambda, z\in \Gamma$, and that $D$ and $D'$ produce the same scattered fields. Further, from the reciprocity relation \cite[Theorem 2.1]{Qin}, we see that for all $z\in \Gamma$ and $x\in D_0= D\cap D'$,
	\begin{align*}
		u_D^s(z;x)=u_{D'}^s(z; x).
	\end{align*}
	Further, the definition of admissible  curve $\Gamma$  and the  reciprocity relation lead to
	\begin{align*}
		u_D^s(x;z)=u_{D'}^s(x; z),\ \forall x,z\in  D_0.
	\end{align*}
	Then, by a similar discussion in Theorem 5.6 in \cite{Colton}, we have $D=D'$ and $\Lambda=\partial D$.
\end{proof}

Now, the indicator functions for a sound-soft or a sound-hard cavity are as follows
\begin{equation}\label{cavityindicator1}
	I_{\rm s}(x)=\int_{\Gamma}|u(x; z)| \mathrm{d}s(z),
\end{equation}
\begin{equation}\label{cavityindicator2}
	I_{\rm h}(x) =\int_{\Gamma}|\nabla u(x; z)\cdot \nu(x; z_0)| \mathrm{d}s(z), 
\end{equation}
where $z_0\in \Gamma$ and $\nu(x; z_0)\cdot \nabla u(x; z_0)=0$. In the next subsection, we will approximate $u(x; z)$ by the Fourier-Bessel functions with the measurements on $\partial B_R$.

\subsection{The Fourier-Bessel approximation}
We assume the scattered fields can be extended analytically to $B_\rho\supseteq D$ and approximate the scattered fields $u^s(x; z, k)$ for every $z\in \Gamma$ by the expansion of Fourier-Bessel functions
\begin{equation}
	u^s_N (x; z)=\sum_{n=-N}^N \hat{u}_n(z) \frac{J_n(kr)}{J_n(k\rho)} \mathrm{e}^{\mathrm{i} n\theta}, \quad r<\rho, \label{Fourier-Bessel}
\end{equation}
with the Fourier coefficients
$$
\hat{u}_n(z)=\frac{1}{2\pi} \int_0^{2\pi} u^s(\rho, \theta; z)\mathrm{e}^{-\mathrm{i} n\theta} \mathrm{d}\theta, 
$$
where $J_n(t)$ is the Bessel function of the first kind of order $n$, under the polar coordinates $(r,\theta): x=r(\cos\theta,\sin\theta)$. The parameters $\hat{u}_n(z)$ can be determined by the measurements $u^s(x; z, k)=u(x; z)-u^i(x; z)$ on $\partial B_R$ by
\begin{equation}\label{FBcoefficient}
	\hat{u}_n(z)=\frac{J_n(k\rho)}{2\pi J_n(kR)} \int_0^{2\pi} u^s(R, \theta; z)\mathrm{e}^{-\mathrm{i} n\theta} \mathrm{d}\theta.
\end{equation}
Due to measured noisy data $u^{s,\delta}\in L^2(\partial B_R)$ satisfying $\|u^{s,\delta}-u^{s}\|_{L^2(\partial B_R)}\leq \delta \|u^{s}\|_{L^2(\partial B_R)}$ with $0<\delta<1$,  the expansion of Fourier-Bessel functions is as follows
\begin{equation}\label{Noisy-Fourier-Bessel}
	u_N^{s,\delta} (x; z)=\sum_{n=-N}^N  \hat{u}_n^\delta(z)\frac{J_n(kr)}{ J_n(k R)}
	 \mathrm{e}^{\mathrm{i} n\theta}, \quad R<r<\rho,
\end{equation}
where
$$
	\hat{u}_n^\delta(z)=\frac{1}{2\pi} \int_0^{2\pi} u^{s,\delta}(R, \theta; z)\mathrm{e}^{-\mathrm{i} n\theta} \mathrm{d}\theta.
$$

Now, we give the following approximation results.
\begin{lemma}\label{Lemma5.1}
	For $z>0$ and $n\in\mathbb{N}$ such that $n>[\max\{ \frac{3}{10} z^2-1,1\}]$, we have
	\begin{equation}\label{estimate7}
		\frac{1}{6} \leq\frac{2^n n!\,| J_n(z)|}{z^n} \leq 1.
	\end{equation}
\end{lemma}

\begin{proof}
	From the definition of the Bessel functions of the first kind
	\begin{align*}
		J_n(z) & = \sum_{p=0}^{\infty}\frac{(-1)^p}{p!(n+p)!}\left(\frac{z}{2}\right)^{n+2p}\\
		& = \frac{z^n}{2^n n!}\left(1-\frac{\left(\frac{z}{2}\right)^2}{n+1}
		+\sum_{p=2}^{\infty}(-1)^p\frac{\left(\frac{z}{2}\right)^{2p}}{p!(n+1)\cdots(n+p)}\right),
	\end{align*}
	we see that for $n\ge \frac{3}{10} z^2-1$,
	$$
		J_n(z) \geq \frac{z^n}{2^n n!}\left(1-\frac{\left(\frac{z}{2}\right)^2}{n+1} \right)
		\geq \frac{1}{6}\frac{z^n}{2^n n!}
	$$
	and
	$$
		J_n(z) \leq\frac{z^n}{2^n n!}.
	$$
	The proof is complete.
\end{proof}

\begin{theorem}\label{thm:estimates_cavity}
	Let $\mu=\mathrm{dist}(\partial B_\rho, \partial D)$ and $N\geq N_0:=[\max\{ \frac{3}{10}(k\rho)^2-1,1\}]$. Then there exist positive constants $K_1, \cdots, K_6$, such that 
	\begin{align}
		\left \|u^s-u_N^{s,\delta}\right\|_{L^2(\partial B_R)} & \le K_1 \sigma_1^{-N}+K_2 N^{1/2}\delta,\label{estimate8} \\ 
		\left \|u^s-u_N^{s,\delta}\right\|_{L^2(\partial D)} & \le K_3\sigma_2^{-N}+K_4\sigma_3^{N}\delta , \label{estimate9} \\
		 \left \|\partial_{\nu} u^s-\partial_{\nu} u_N^{s, \delta}\right\|_{L^2(\partial D)} & \le K_5 N\sigma_2^{-N}+K_6 N\sigma_3^N\delta,  \label{estimate91}
	\end{align}
	where $\sigma_1=\frac{\rho}{R}, \sigma_2=\frac{\rho}{\rho-\mu}, \sigma_3=\frac{\rho-\mu}{R}$. Moreover, for $0<\delta<\sigma_1^{-N_0}$, take $N= \left[\frac{1}{\ln \sigma_1}\ln \frac{1}{\delta}\right]$, then the following results hold
	\begin{align}
		\left \|u^s-u_N^{s, \delta}\right\|_{L^2(\partial B_R)} & \le K_1 \delta+\frac{K_2}{\sqrt{\ln \sigma_1}}   \delta^{1/2},\label{estimate10} \\ 
		\left \|u^s-u_N^{s, \delta}\right\|_{L^2(\partial D)} & \le (K_3+K_4) \delta^{\frac{\ln \sigma_2}{\ln \sigma_1}}, \label{estimate11} \\
		\left \|\partial_{\nu}u^s-\partial_{\nu}u_N^{s, \delta}\right\|_{L^2(\partial D)} & \le
		\frac{K_5+K_6}{\ln \sigma_1} \delta^{\frac{\ln \sigma_2}{\ln \sigma_1}}|\ln \delta|. \label{estimate12}
	\end{align}
\end{theorem}

\begin{proof}
	From \cref{Fourier-Bessel} and \cref{estimate7}, it can be seen that
	\begin{align*}
		\left \|u^s-u^s_N\right\|^2_{L^2(\partial B_R)}&= 2\pi R  \sum_{|n|>N} \left|\frac{J_n(kR)}{J_n(k\rho)}\right|^2|\hat{u}_n(z)|^2\\
		&\le 2\pi R \sum_{|n|>N}36 \left(\frac{R}{\rho}\right)^{2|n|}|\hat{u}_n(z)|^2\\
		&\le 72\pi R\left(\frac{R}{\rho}\right)^{2N}\|u^s\|_{L^2(\partial B_\rho)}^2.
	\end{align*}
	Further, by \cref{Fourier-Bessel}-\cref{Noisy-Fourier-Bessel}, we  have
	\begin{align*}
		\left \|u^s_N-u_N^{s,\delta}\right\|_{L^2(\partial B_R)}^2
		&= \frac{1}{2\pi R}\sum_{n=-N}^N  \left |\left\langle u^s-u^{s,\delta}, \mathrm{e}^{\mathrm{i} n\theta}\right\rangle_{\partial B_R} \right|^2 \\
		& \le (2N+1)\delta^2 \|u^{s}\|_{L^2(\partial B_R)}^2.
	\end{align*}
	And thus,
	\begin{align*}
		\left \|u^s-u_N^{s,\delta}\right\|_{L^2(\partial B_R)} & \le \left \|u^s-u^s_N\right\|_{L^2(\partial B_R)}+ \left \|u^s_N-u_N^{s,\delta}\right\|_{L^2(\partial B_R)}\\
		& \le K_1 \sigma_1^{-N}+K_2 N^{1/2}\delta,
	\end{align*}
	where $\sigma_1=\frac{\rho}{R}, K_1=6\sqrt{2\pi R}\|u^s\|_{L^2(\partial B_\rho)}$ and
	$K_2=\sqrt{3}\|u^s\|_{L^2(\partial B_R)}$. This leads to estimate \cref{estimate8}.
	
	From \cref{estimate7}, it is readily to see that for $x\in \partial D$,
	\begin{align*}
		\left |u^s(x; z)-u^s_N(x; z)\right| & \le  \sum_{|n|>N} \left |\frac{J_n(kr)}{J_n(k\rho)} \right| |\hat{u}_n(z)|\\
		& \le 6 \sum_{|n|>N} \left(\frac{r}{\rho}\right)^{|n|} |\hat{u}_n(z)| \\
		& \le 6\sum_{|n|>N} \left(\frac{\rho-\mu}{\rho}\right)^{|n|} |\hat{u}_n(z)| \\
		& \le 6\sigma_2^{-N}\left(\sum_{|n|\geq1}\sigma_2^{-2|n|}\right)^{1/2} \left(\sum_{|n|>N}|\hat{u}_n(z)|^2\right)^{1/2}\\
		& \le  \widetilde{K}_3\sigma_2^{-N} \|u^s\|_{L^2(\partial B_\rho)},
	\end{align*}
	where $\mu=\mathrm{dist}(\partial B_\rho, \partial D), \sigma_2=\frac{\rho}{\rho-\mu}$  and $\widetilde{K}_3=6\sqrt{2}/\sqrt{\sigma_2^2-1} $. This implies
	\begin{equation}\label{eqn8}
		\left \|u^s-u^s_N\right\|_{L^2(\partial D)}\leq K_3\sigma_2^{-N},
	\end{equation}
	where $K_3=\widetilde{K}_3\|u^s\|_{L^2(\partial B_\rho)}|\partial D|^{1/2}$. Similarly, by using \cref{Noisy-Fourier-Bessel} and \cref{estimate7}, we have that for $x\in \partial D$,
	\begin{align*}
		\left |u^s_N(x; z)-u_N^{s,\delta}(x; z)\right|
		\leq & \frac{1}{2\pi R}\sum_{n=-N}^N  \left | \frac{J_n(kr)}{J_n(k R)} \right|
		\left |\left\langle u^s-u^{s,\delta}, \mathrm{e}^{\mathrm{i} n\theta}\right\rangle_{\partial B_R} \right|\\
		\leq & \frac{1}{\sqrt{2\pi R}}\sum_{n=-N}^N \left | \frac{J_n(kr)}{ J_n(k R)}  \right|
		\delta \|u^{s}\|_{L^2(\partial B_R)}\\
		\leq & \frac{\|u^s\|_{L^2(\partial B_R)}}{\sqrt{2\pi R}}\delta\sum_{n=-N}^N \frac{4}{3} \left(\frac{\rho-\mu}{R}\right)^{|n|}\\
		\leq & \widetilde{K}_4\delta \sigma_3^{N+1},
	\end{align*}
	where $\sigma_3=\frac{\rho-\mu}{R}$ and $\widetilde{K}_4=\frac{6\sqrt{2}\|u^{s}\|_{L^2(\partial B_R)}}{3\sqrt{\pi R}(\sigma_3-1)}$. This, together with \cref{eqn8}, yields
	\begin{align*}
		\left \|u^s-u_N^{s,\delta}\right\|_{L^2(\partial D)}\leq & \left \|u^s-u^s_N\right\|_{L^2(\partial D)}
		+ \left \|u^s_N-u_N^{s,\delta}\right\|_{L^2(\partial D)}\\
		\leq & K_3\sigma_2^{-N}+K_4\delta \sigma_3^N,
	\end{align*}
	where $K_4=\widetilde{K}_4\sigma_3|\partial D|^{1/2}$, which implies \cref{estimate9}.
	
	From the formula $J_n'(t)=J_{n-1}(t)-n J_n(t)/t$, it can be seen that
	$$
		\left|\nabla J_n(kr)\right|\leq k \left|J_{n}'(kr)\right|\leq k\left|J_{n-1}(kr)\right|+\frac{|n|}{r}\left| J_n(kr)\right|,
	$$
	which, together with $\left|\nabla  \mathrm{e}^{\mathrm{i}n \theta}\right|=|n|/r$, implies
	\begin{equation}\label{gradient1}
		\left|\nabla \left(J_n(kr)\mathrm{e}^{\mathrm{i}n \theta}\right)\right|\leq k\left| J_{n-1}(kr)\right|+\frac{2|n|}{r}\left| J_n(kr)\right|.
	\end{equation}
	Now, by using estimate \cref{estimate7}, we deduce that for $x\in \partial D$,
	\begin{align*}
		 \left |\frac{\partial}{\partial \nu}u^s(x; z) -\frac{\partial}{\partial \nu}u_N^s(x; z)\right|
		 & \le  \sum_{|n|>N} |\hat{u}_n(z)| \left |\frac{\nabla \left(J_n(kr)\mathrm{e}^{\mathrm{i}n \theta}\right)}{J_n(k\rho)} \right| \\
		& \le\sum_{|n|>N} |\hat{u}_n(z)| \left(k\left |\frac{J_{n-1}(kr)}{J_n(k\rho)} \right|+\frac{2|n|}{r}\left |\frac{J_n(kr)}{J_n(k\rho)} \right|\right) \\
		& \le\frac{24}{R}\sum_{|n|>N}|n|\left(\frac{\rho-\mu}{\rho}\right)^{|n|} |\hat{u}_n(z)| \\
		& \le \frac{24}{R}N\sigma_2^{-N}\left(\sum_{|n|\geq1}n^2\sigma_2^{-2|n|}\right)^{1/2} \left(\sum_{|n|>N}|\hat{u}_n(z)|^2\right)^{1/2}\\
		& \le  \widetilde{K}_5 N\sigma_2^{-N} \|u^s\|_{L^2(\partial B_\rho)},
	\end{align*}
	where $\widetilde{K}_5=\frac{24}{R}\left(\sum_{|n|\geq1}n^2\sigma_2^{-2|n|}\right)^{1/2} $. This implies
	\begin{equation}\label{eq:estimate_K5}
		\left \|\frac{\partial}{\partial \nu} u^s(x; z)-\frac{\partial}{\partial \nu}u_N^s(x; z)\right\|_{L^2(\partial D)}\le  K_5 N\sigma_2^{-N},
	\end{equation}
	where $K_5=\widetilde{K}_5\|u^s\|_{L^2(\partial B_\rho)}|\partial D|^{1/2}$. Further, from  estimates \cref{estimate7} and \cref{gradient1}, we see that for $x\in \partial D$,
	\begin{align*}
		\left |\frac{\partial}{\partial \nu}u_N^s(x; z)-\frac{\partial}{\partial \nu}u_N^{s, \delta}(x; z)\right|
		& \le \frac{1}{2\pi R}\sum_{n=-N}^N  \left |\frac{\nabla \left(J_n(kr)\mathrm{e}^{\mathrm{i}n \theta}\right)}{J_n(kR)} \right|
		\left |\left\langle u^s-u^{s,\delta}, \mathrm{e}^{\mathrm{i} n\theta}\right\rangle_{\partial B_R} \right| \\
		& \le \frac{1}{\sqrt{2\pi R}}\sum_{n=-N}^N \left(k\left|\frac{J_{n-1}(kr)}{J_n(kR)} \right|+\frac{2|n|}{r}\left|\frac{J_n(kr)}{J_n(kR)} \right|\right)
		\delta \|u^{s}\|_{L^2(\partial B_R)}\\
		& \le \frac{\|u^{s}\|_{L^2(\partial B_R)}}{\sqrt{2\pi R}}\delta\sum_{n=-N}^N \frac{24|n|}{R} \left(\frac{\rho-\mu}{R}\right)^{|n|}\\
		& \le \frac{12\sqrt{2}\|u^s\|_{L^2(\partial B_R)}}{R\sqrt{\pi R}}\delta N\sum_{n=-N}^N  \left(\frac{\rho-\mu}{R}\right)^{|n|}\\
		& \le \widetilde{K}_6\delta N\sigma_3^{N+1},
	\end{align*}
	where $\widetilde{K}_6=\frac{24\sqrt{2}\|u^s\|_{L^2(\partial B_R)}}{R\sqrt{\pi R}(\sigma_3-1)}$. This, together with \cref{eq:estimate_K5}, yields
	 \begin{align*}
		\left\|\partial_{\nu}u^s-\partial_{\nu}u_N^{s, \delta}\right\|_{L^2(\partial D)}& \le
		\left \|\partial_{\nu}u^s-\partial_{\nu}u_N^s\right\|_{L^2(\partial D)}
		+ \left \|\partial_{\nu}u_N^s-\partial_{\nu}u_N^{s, \delta}\right\|_{L^2(\partial D)}\\
		& \le K_5 N\sigma_2^{-N}+K_6\delta N\sigma_3^{N},
	\end{align*}
	where $K_6=\widetilde{K}_6\sigma_3|\partial D|^{1/2}$, which implies \cref{estimate91}.
	
	For $0<\delta<\sigma_1^{-N_0}$, from $N=\left[\frac{1}{\ln \sigma_1}\ln \frac{1}{\delta}\right]$, we see $N\geq N_0$. Substituting $N=\left[\frac{1}{\ln \sigma_1}\ln \frac{1}{\delta}\right]$ into \cref{estimate8}, \cref{estimate9}, and using  $\ln t<t$ for $t>1$, $\sigma_3<\sigma_1$, we derive the estimates \cref{estimate10} and \cref{estimate11}. The proof is complete.
\end{proof}

With the approximation of the scattered fields $u^s(x; z)$ by $u_N^{s,\delta} (x; z)$ in \cref{Noisy-Fourier-Bessel}, we achieve the approximating indicator functions $\mathcal{I}_{{\rm s},N}(x)$ and $\mathcal{I}_{{\rm h},N}(x)$, which have the same forms as those in \cref{obstacleindicator4}-\cref{obstacleindicator5}. And by analogous arguments in Theorem \ref{thm:estimates_indicator_obstacle}, we have the following stability result on the indicator functions for the inverse cavity scattering.
\begin{theorem}\label{thm:estimates_indicator_cavity}
	Let $\sigma_1, \sigma_2$ and $N_0$ be defined in \cref{thm:estimates_cavity}. Then for $0<\delta<\sigma_1^{-N_0}$ and $N=\left[\frac{1}{\ln \sigma_1}\ln \frac{1}{\delta}\right]$, there exist positive constants $K_{\rm s}$ and $K_{\rm h}$, such that
	\begin{align*}
		\mathcal{I}_{{\rm s},N}(x) & \le K_{\rm s} \delta^{\beta}, \\ 
		\mathcal{I}_{{\rm h},N}(x) & \le K_{\rm h} \delta^{\beta}|\ln \delta|,
	\end{align*}
	where $\beta= \ln \sigma_2/\ln \sigma_1$.
\end{theorem}

\section{Numerical algorithm and examples}\label{sec:example}

Based on the indicator functions in the previous sections, the direct imaging scheme for reconstructing the shape of scatterer is given in the following {\bf Algorithm}.
\begin{table}[htp]
	\centering
	\begin{tabular}{cp{.8\textwidth}}
		\toprule
		\multicolumn{2}{l}{{\bf Algorithm:}\quad Reconstruction by the Fourier-Bessel approximated imaging scheme} \\
		\midrule
		{\bf Step 1} & Given a monochromatic frequency or a spectrum of frequencies,  collect the corresponding noisy near-field data due to the point sources and the obstacle/cavity with sound-soft or sound-hard boundary condition. \\
		{\bf Step 2} & Choose an approximate truncation $N$ and compute the approximate scattered field by the Fourier-Hankel expansion \cref{Fourier-Hankel} or Fourier-Bessel expansion \cref{Fourier-Bessel}.\\
		{\bf Step 3} & Select a suitable imaging mesh  $\mathcal{T}$ covering the target scatterer.  In terms of the boundary condition, for each imaging point $x\in \mathcal{T}$, evaluate the perturbed indicator functions $\mathcal{I}_{{\rm s}, N}(x)$ or $\mathcal{I}_{{\rm h}, N}(x)$, according to\cref{obstacleindicator4,obstacleindicator5} for the obstacle or \cref{cavityindicator1,cavityindicator2} for the cavity.\\
		{\bf Step 4} & The boundary of the scatterer can be recovered as the zeros of the monochromatic indicator functions or their superpositions with respect to the frequencies. \\
		\bottomrule
	\end{tabular}
\end{table}

In this section, we present several numerical examples of the inverse obstacle/cavity scattering problems in two dimensions to illustrate the effectiveness and applicability of our algorithm. In our numerical experiments, the boundary of the model scatterer is represented by the parametric form
$$
	x(t) = (x_1(t), x_2(t)), \quad 0\leq t<2\pi
$$
which is a simple smooth and closed curve. We compute the synthetic near-field data by solving the forward problems via the boundary integral equation method \cite{Colton}.  The exact boundary curves of the scatterers are listed in the parametric form and plotted in \cref{fig:model_shapes}. 
\begin{figure}
	\centering	
	\subfigure[]{\includegraphics[width=0.3\textwidth]{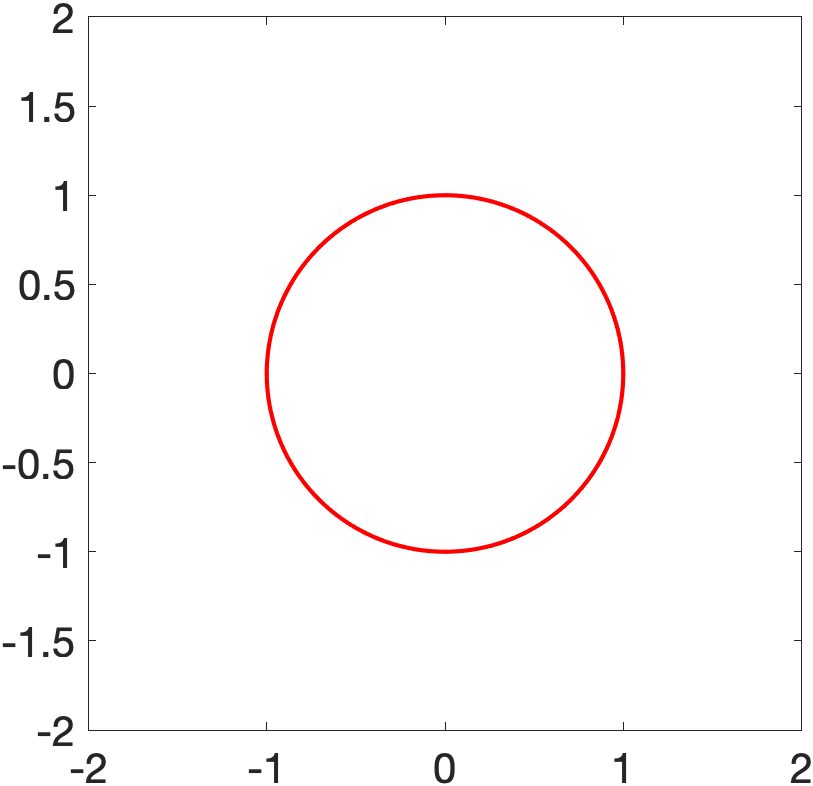}}\quad
	\subfigure[]{\includegraphics[width=0.3\textwidth]{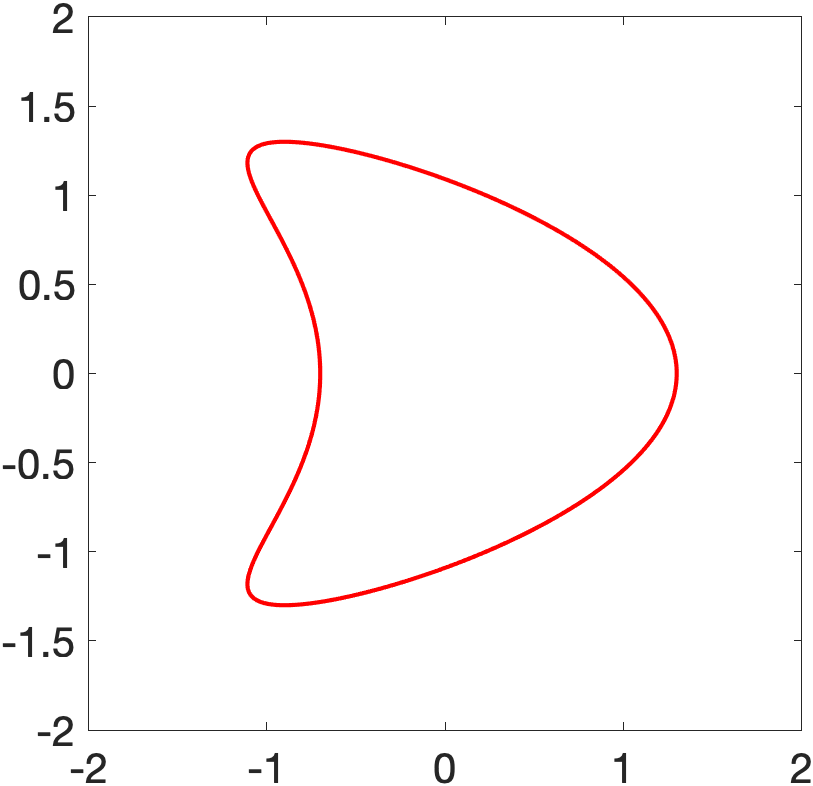}}\quad
	\subfigure[]{\includegraphics[width=0.3\textwidth]{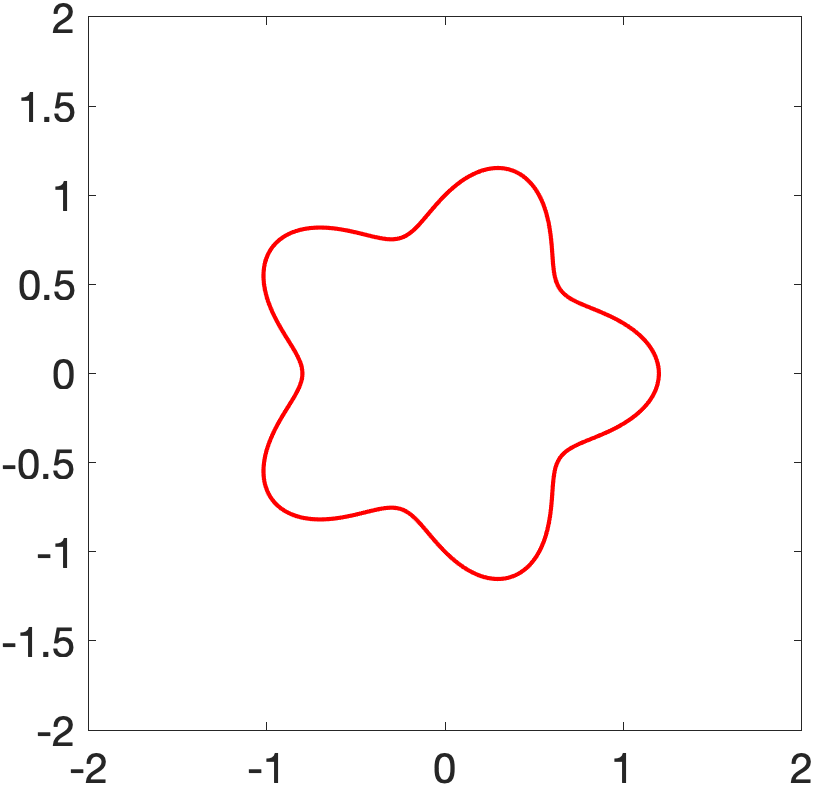}}
	\caption{Model scatterers. (a) circle: $(\cos t, \sin t)$; (b) kite: $(\cos t+0.6\cos 2t-0.3, 1.3\sin t)$; (c) starfish: $(1+0.2\cos 5t)(\cos t,\sin t)$.}\label{fig:model_shapes}
\end{figure}

With the purpose of testing the stability of the proposed method, we further add some random noise to the computed scattered data by
$$
u^{s,\delta}=u^s+\delta r_1 |u^s|\mathrm{e}^{\mathrm{i}\pi r_2}
$$
where $r_1, r_2$ are two uniformly distributed random number ranging from $-1$ to $1$, and $\delta>0$ denotes the relative noise level. 

For the ease of visualization and comparison, the reciprocals (multiplicative inverse) of the normalized indicator functions are depicted in the following figures. Consequently, the profile of the scatterer can be qualitatively reconstructed as the significant peak levels in the images. As a rule of thumb, truncation of the Fourier-Bessel expansion is chosen as $N=[|\ln\delta|]+1$ for the exterior problem and $N=[1.5|\ln\delta|]+1$ for the interior problem .

\subsection{Exterior problems}

In this part for the exterior problems, we present the results for imaging the obstacles with sound-soft and sound-hard boundary conditions. We collect the scattered data on the measurement circle centered at the origin with radius 2.2.  In each example, 12 point sources and 128 receivers are equally placed on the measurement circle. These sources and receivers are marked by the small red and blue points, respectively.  Hereafter, the boundary of exact obstacle is denoted by the black dashed lines. The imaging domain is chosen as $[-1.5, 1.5]^2$ with $150\times 150$ equally spaced imaging grid. 

\begin{example}
The goal of the first example is to reconstruct the sound-soft obstacles. Here wavenumbers $k=3$ and $k=6$ are used. The reconstructions are shown in \cref{fig:obstacle_Dirichlet} and it can be seen that the boundary curves are well recovered by the proposed imaging method.
\begin{figure}
	\centering	
	\subfigure[]{\includegraphics[width=0.3\textwidth]{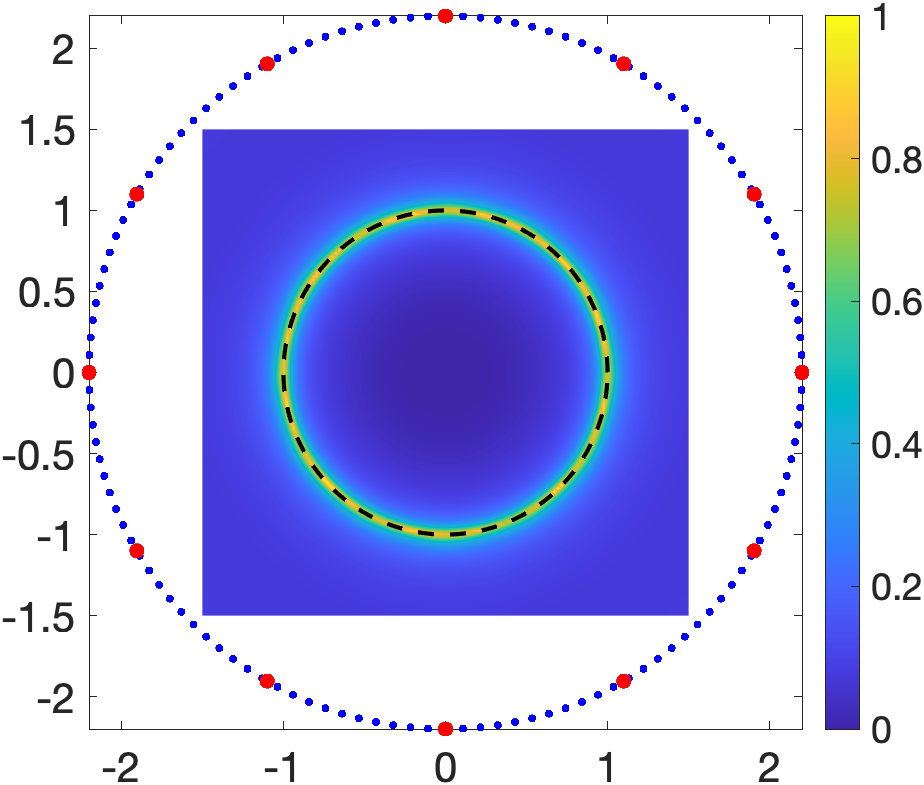}}\quad
	\subfigure[]{\includegraphics[width=0.3\textwidth]{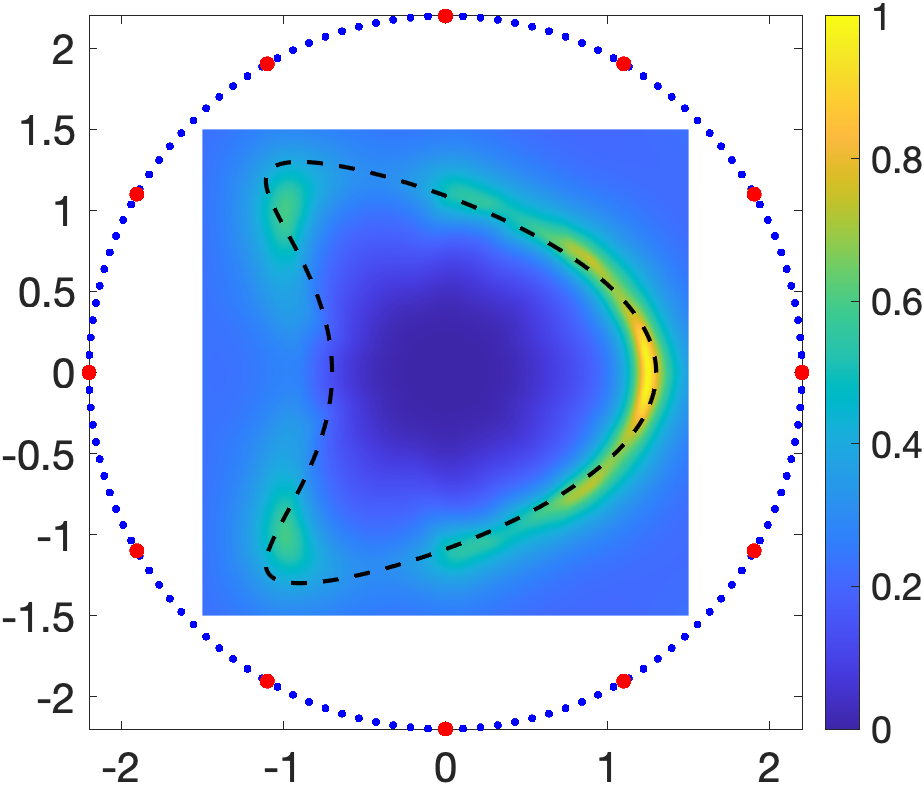}}\quad
	\subfigure[]{\includegraphics[width=0.3\textwidth]{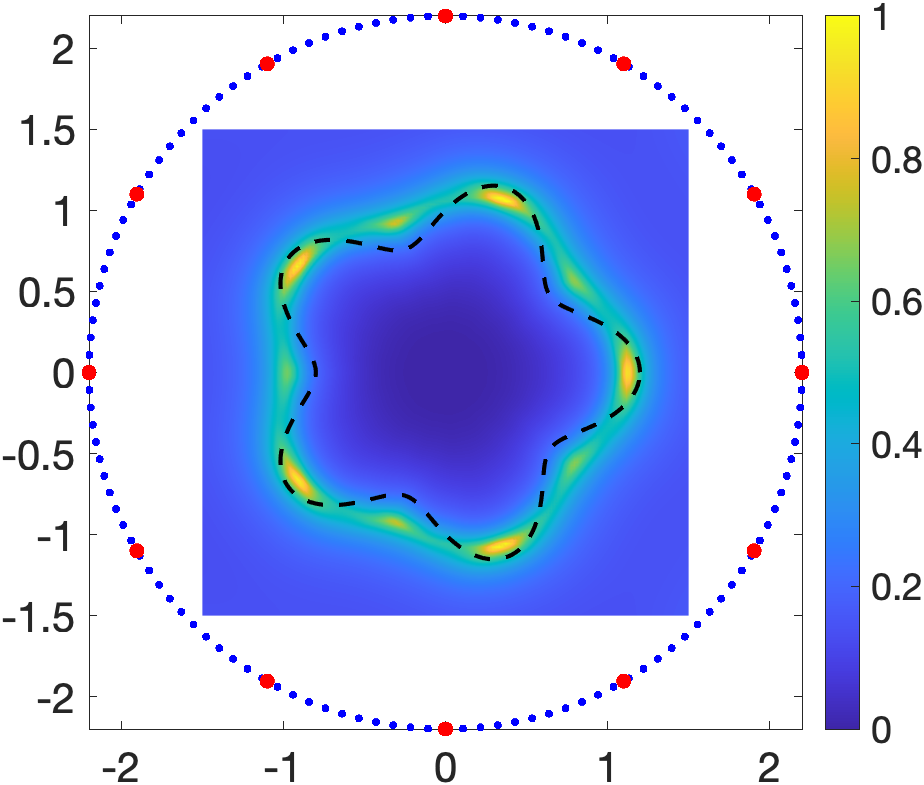}}\\
	\subfigure[]{\includegraphics[width=0.3\textwidth]{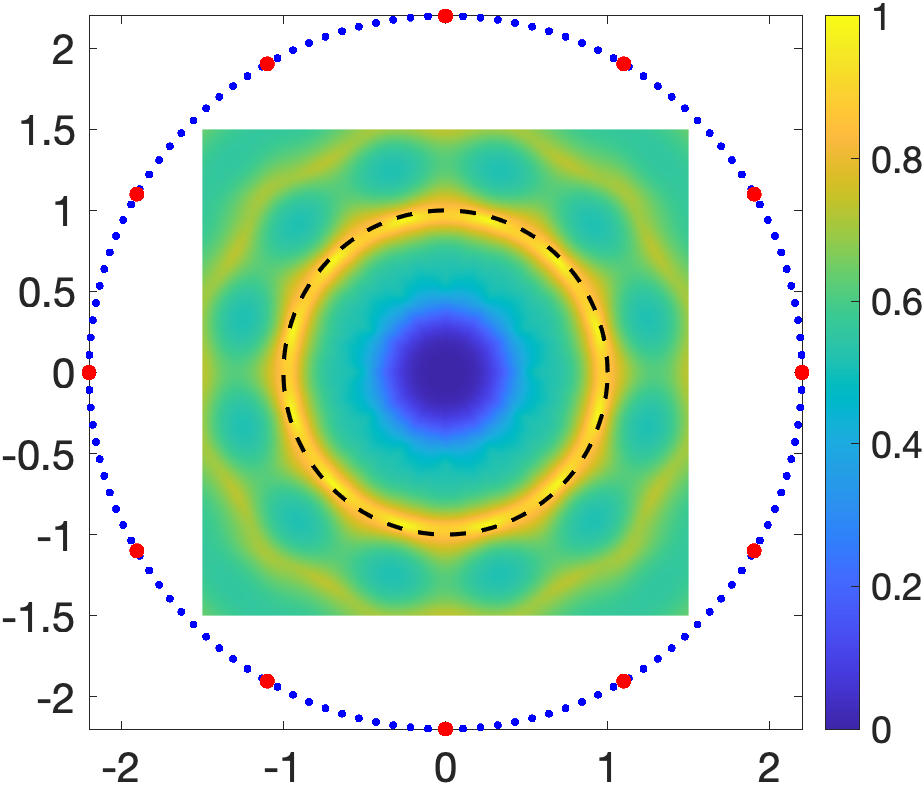}}\quad
    \subfigure[]{\includegraphics[width=0.3\textwidth]{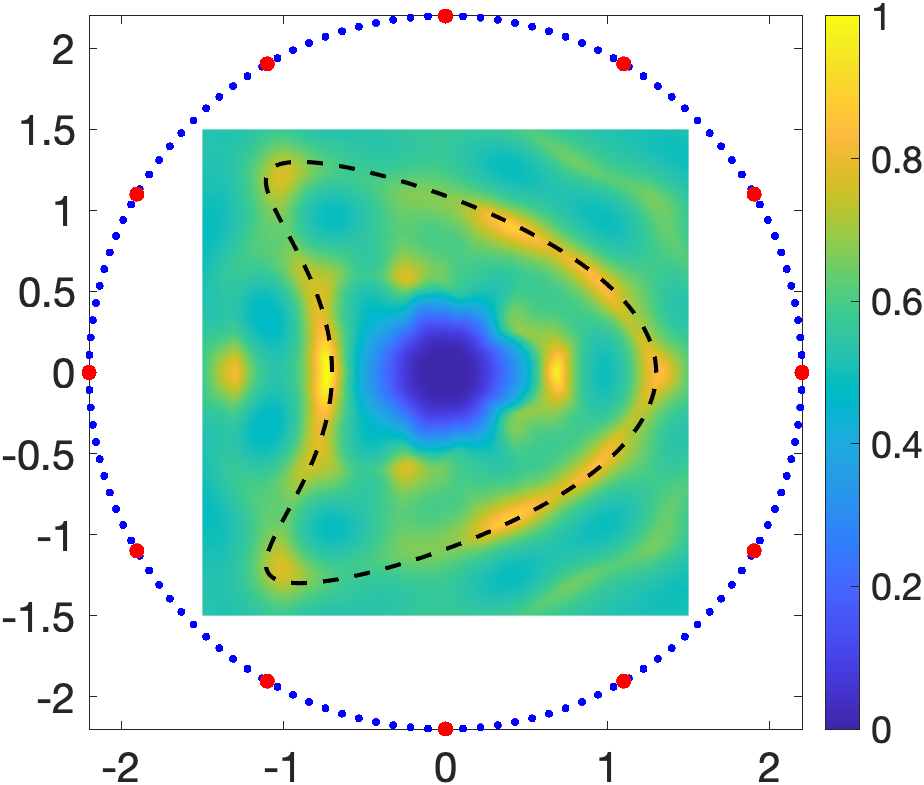}}\quad
    \subfigure[]{\includegraphics[width=0.3\textwidth]{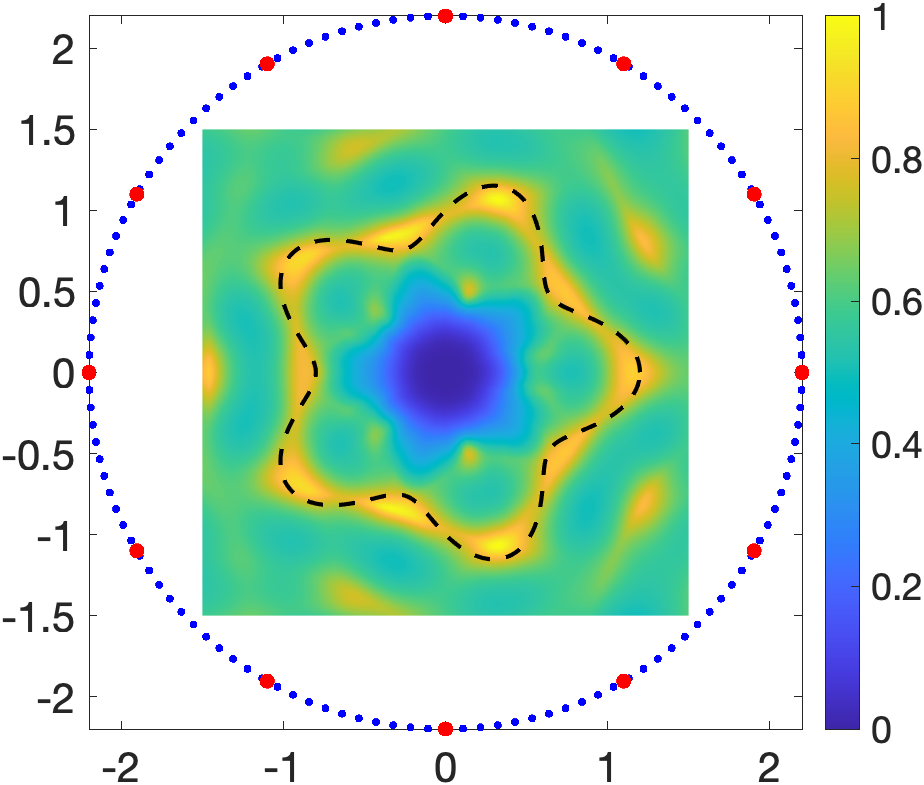}}
	\caption{Reconstruction of sound-soft obstacles with 5\% noise. Top row:  $k=3$; Bottom row: $k=6$. }\label{fig:obstacle_Dirichlet}
\end{figure}
\end{example}

\begin{example}
In the second example, we present the reconstructions of sound-hard obstacles. We respectively choose $k=4$ and $k=5$ for the single-frequency indicator and $k=3+\ell/2, (\ell=0, 1,\cdots, 6)$ for the multi-frequency indicator. Moreover, for the multi-frequency case, the final reconstruction is produced by the superposition of each normalized single-frequency indicators. The results are plotted in \cref{fig:obstacle_Neumann}. 
\begin{figure}
	\centering	
	\subfigure[]{\includegraphics[width=0.3\textwidth]{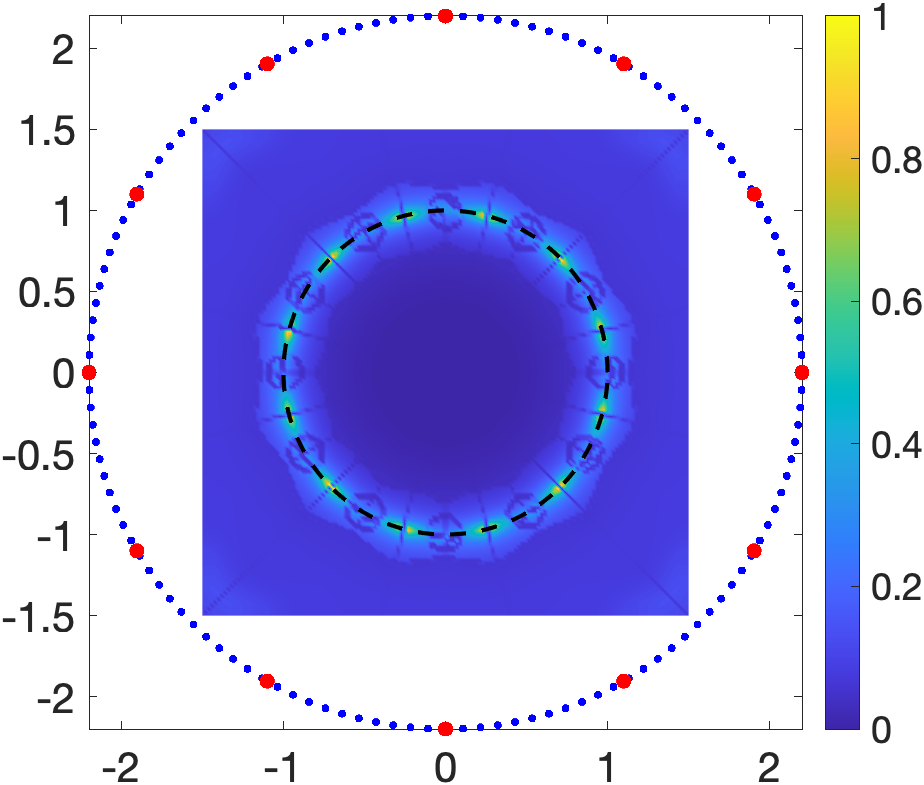}}\quad
	\subfigure[]{\includegraphics[width=0.3\textwidth]{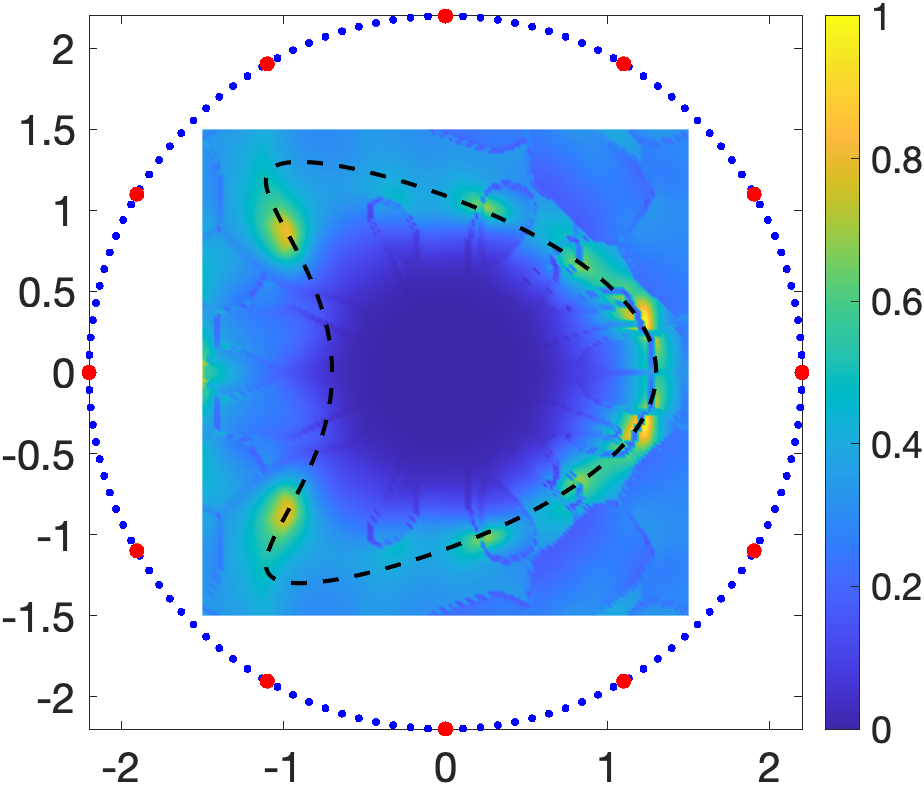}}\quad
	\subfigure[]{\includegraphics[width=0.3\textwidth]{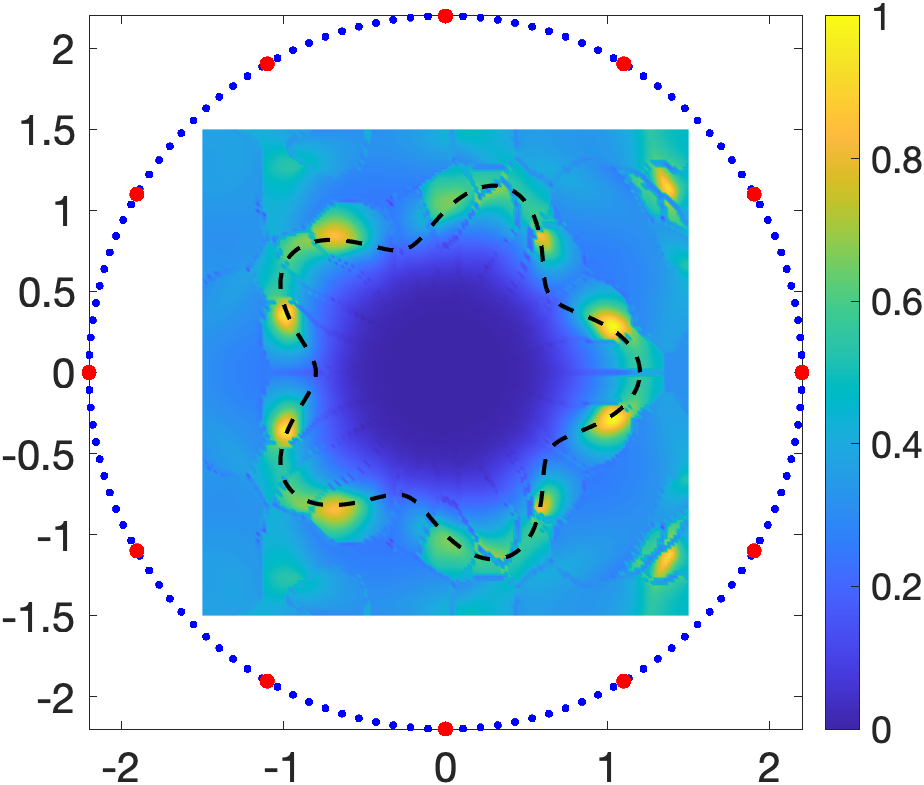}}\\
	\subfigure[]{\includegraphics[width=0.3\textwidth]{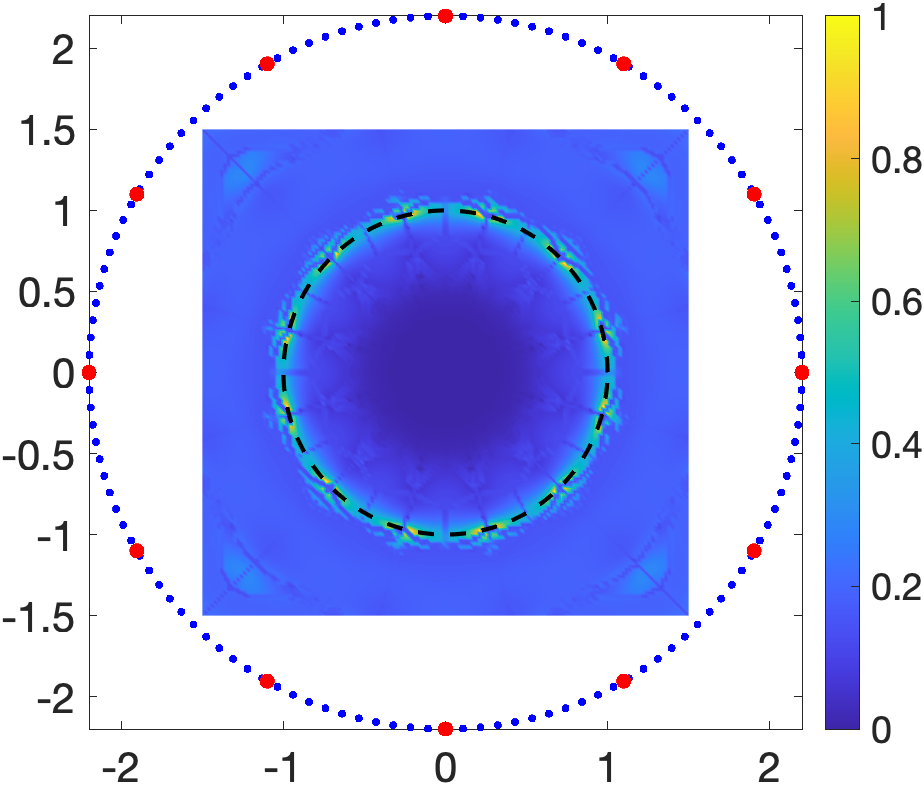}}\quad
	\subfigure[]{\includegraphics[width=0.3\textwidth]{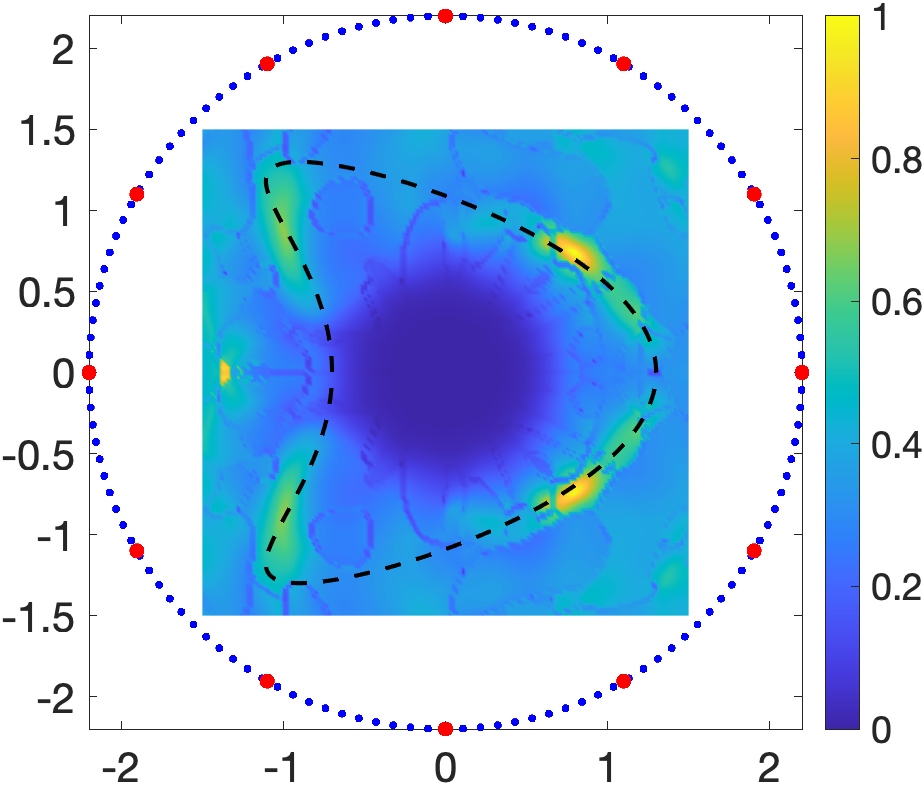}}\quad
	\subfigure[]{\includegraphics[width=0.3\textwidth]{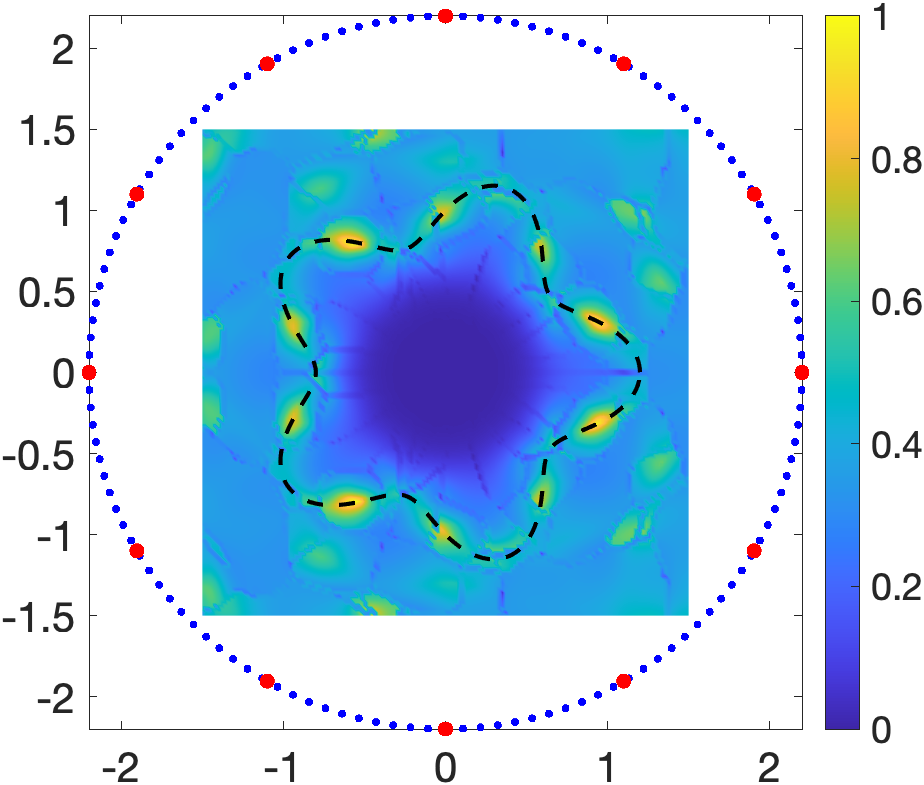}}\\
	\subfigure[]{\includegraphics[width=0.3\textwidth]{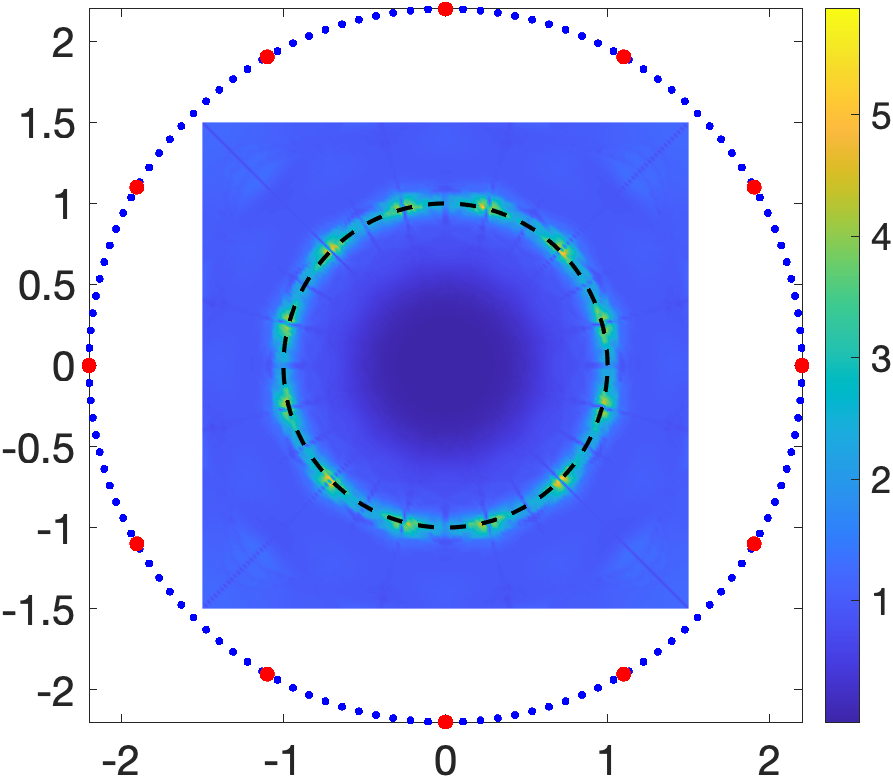}}\quad
	\subfigure[]{\includegraphics[width=0.3\textwidth]{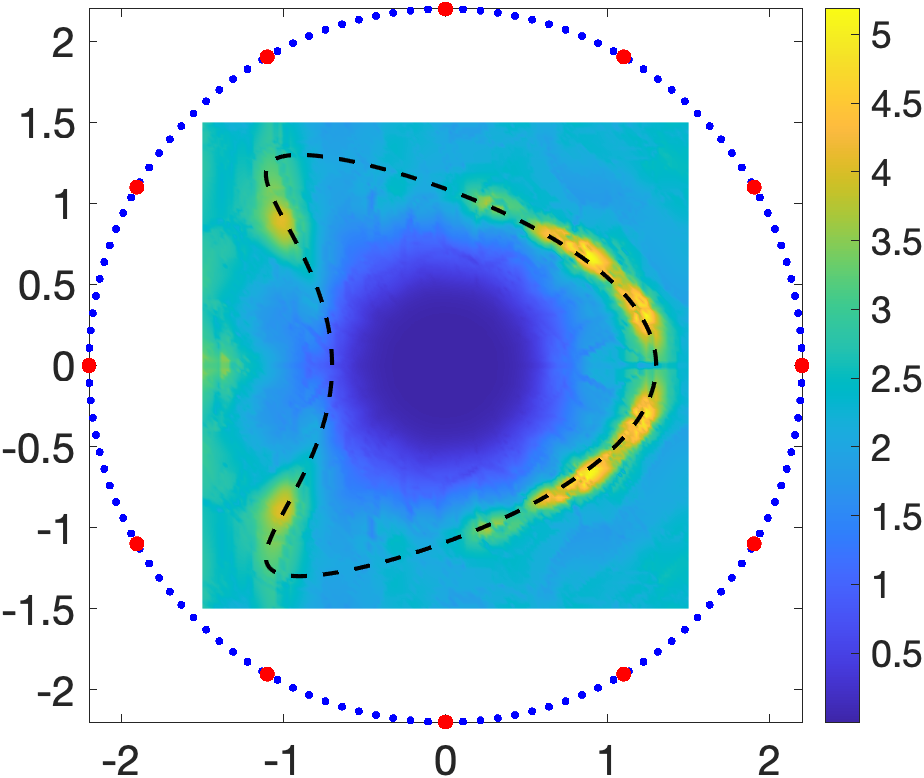}}\quad
	\subfigure[]{\includegraphics[width=0.3\textwidth]{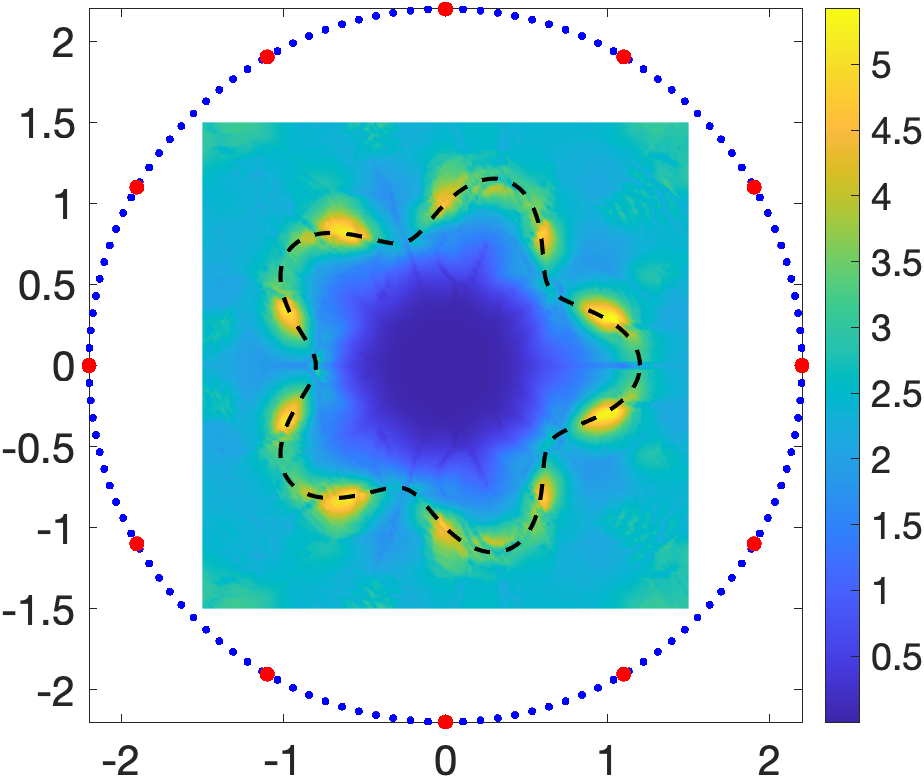}}
	\caption{Reconstruction of sound-hard obstacles with 2\% noise. Top row: $k=4$; Middle row: $k=5$; Bottom row: superposition with multiple frequencies $k=3, 3.5, \cdots, 6$.}
	\label{fig:obstacle_Neumann}
\end{figure}
\end{example}

It deserves noting that, no matter what the boundary type is, the sharp corners facing outside are superiorly reconstructed. This is because these corners are relatively more closer to the sources and receivers. In other words, they could be more adequately illuminated by the point sources in certain sense. We would like to remark that, if the incident wave is alternatively chosen as the plane wave, then the same indicator functions can be defined and our algorithm procedure remains valid. The plane wave incidence would make the scattered information be captured in a more balanced manner and consequently improve the reconstructions. Since the case of plane wave is beyond the current study, the numerical results with plane wave will be reported in a future work.

\subsection{Interior problems}

Similar to the preceding examples for the exterior problems, we present the reconstruction results for the cavities in this subsection. Both the sources and receivers are located on an interior detection curve, which is chosen as the circle centered at the origin with radius 0.5. Without otherwise specified, all the parameters used here are the same as the exterior problems. The imaging domain is chosen as $[-1.5, 1.5]^2$ with $150\times 150$ equally spaced imaging grid excluding the interior of measurement circle. For the sake of comparison, we also present the reconstructions of previous circle, kite- and starfish-shaped scatterers.

\begin{example}
In this example, we consider the recovery of sound-soft cavities with wavenumbers $k=3$ and $k=5$ with $\delta=0.05$. The reconstructions are illustrated in \cref{fig:cavity_Dirichlet}. One can observe that the circle and the starfish are relatively well recovered while two wings of the kite are far less accurate. Physically speaking, this phenomenon is probably due to the fact that the energy of point sources decays rapidly with respect to the increasing of propagating distance and the wavenumber. Hence, the scattered signals become much weaker when any portion of cavity recedes from the sources. As a result, one should not expect to obtain a large amount of information from the far side of cavity and thus the corresponding reconstruction inevitably deteriorates. Similar effects and discussions can be applied to the other examples regardless of boundary conditions.

\begin{figure}
	\centering	
	\subfigure[]{\includegraphics[width=0.3\textwidth]{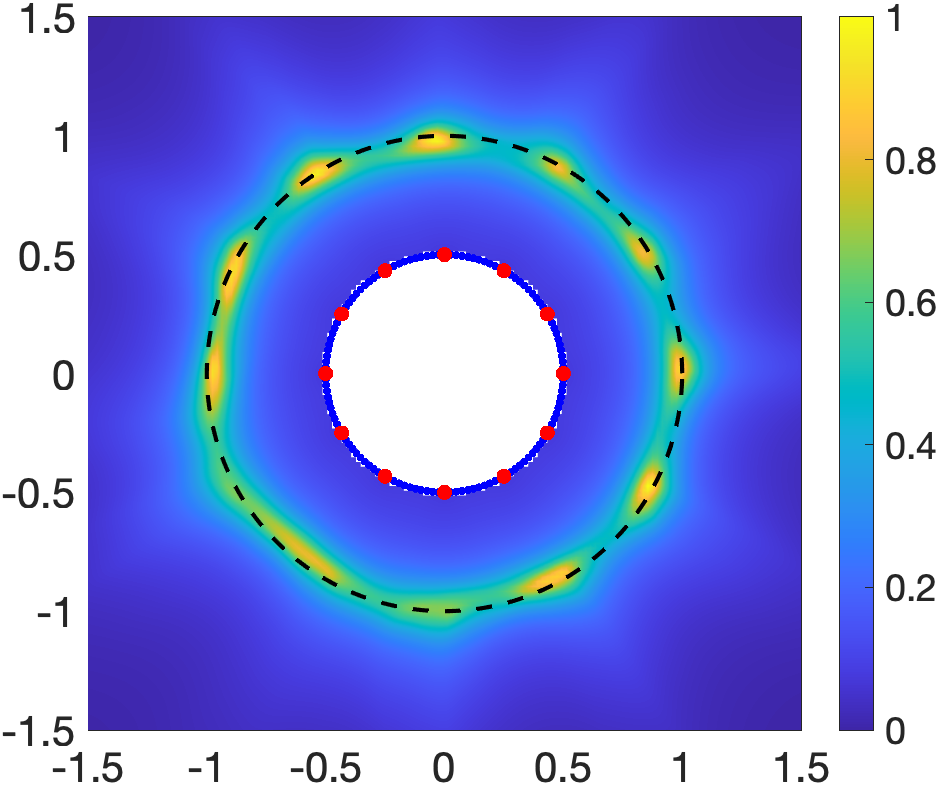}}\quad
	\subfigure[]{\includegraphics[width=0.3\textwidth]{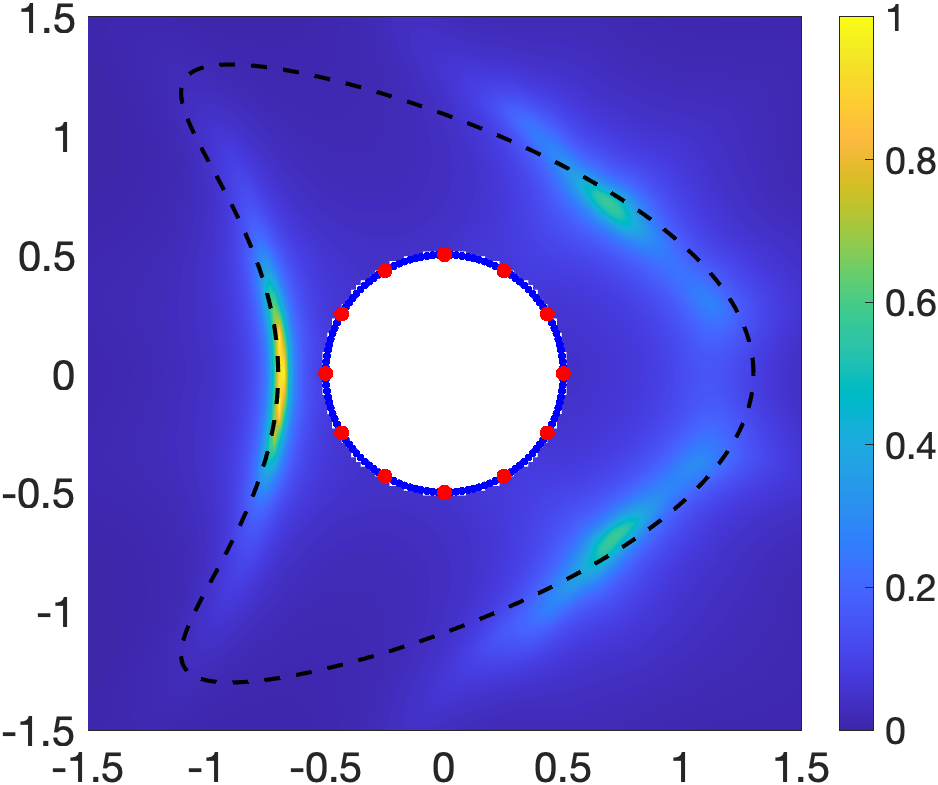}}\quad
	\subfigure[]{\includegraphics[width=0.3\textwidth]{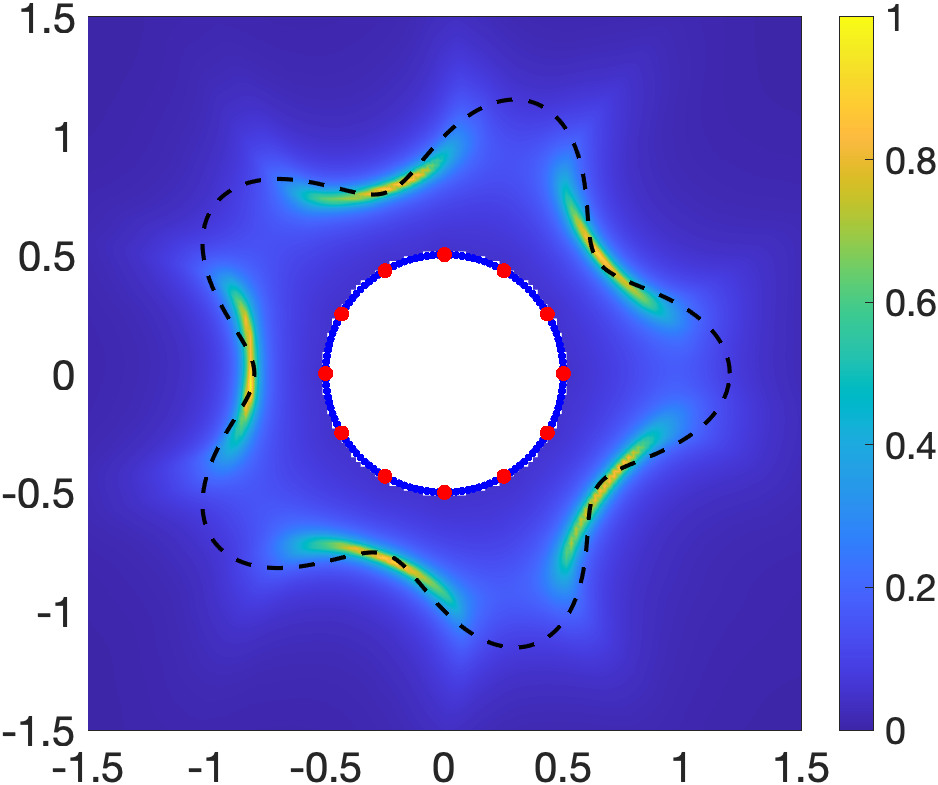}}\\
	\subfigure[]{\includegraphics[width=0.3\textwidth]{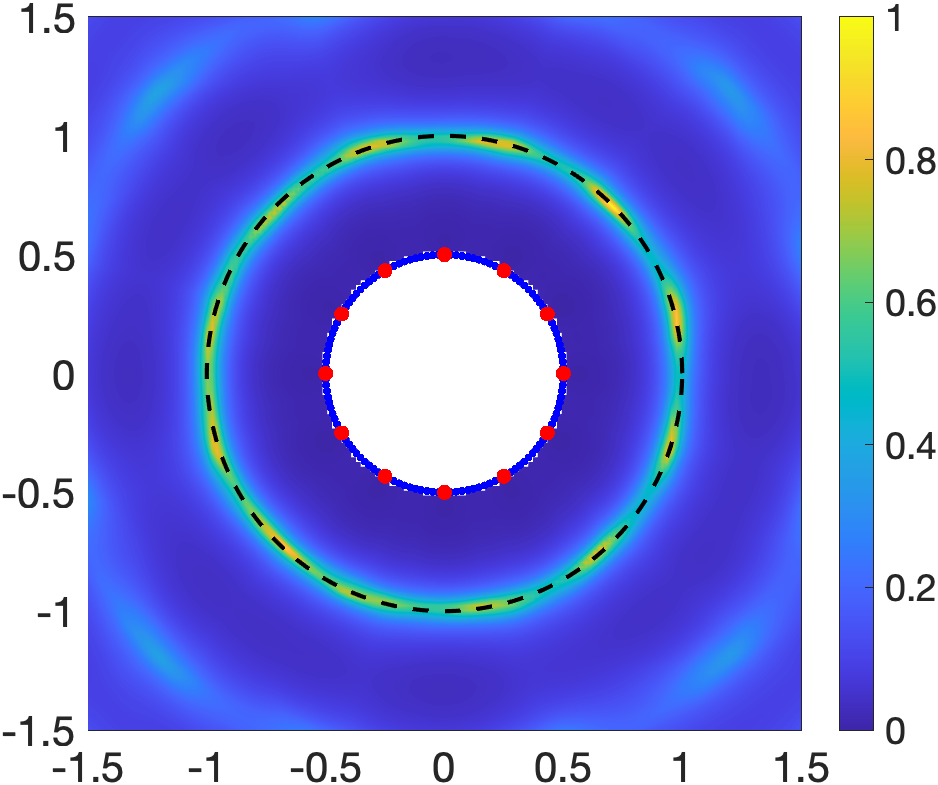}}\quad
    \subfigure[]{\includegraphics[width=0.3\textwidth]{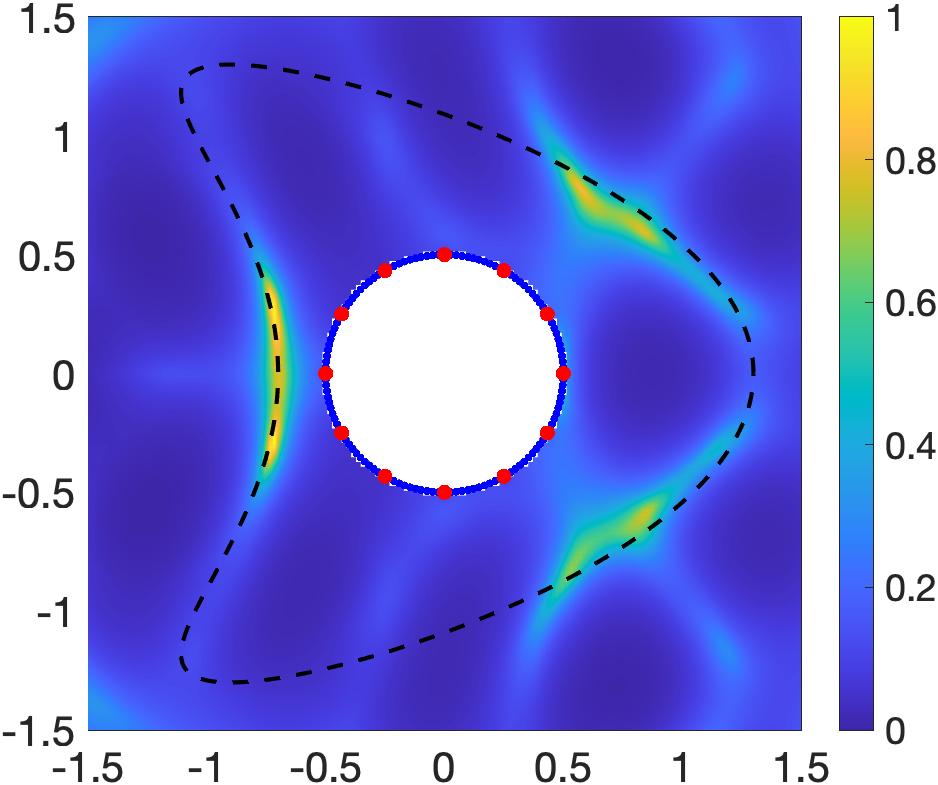}}\quad
    \subfigure[]{\includegraphics[width=0.3\textwidth]{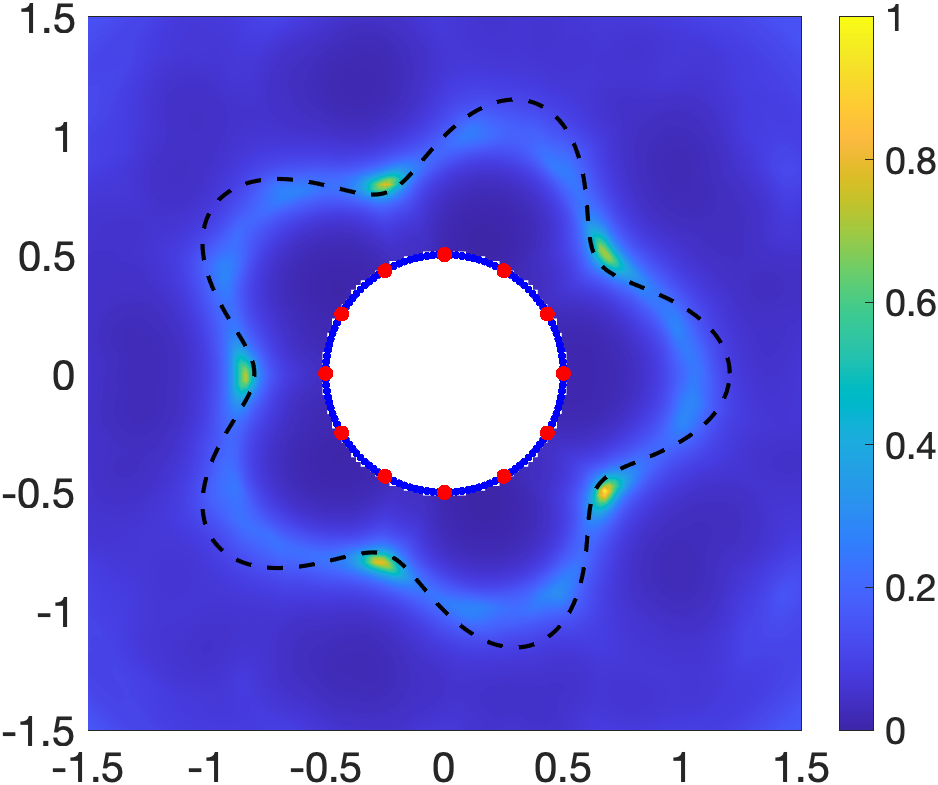}}
	\caption{Reconstruction of sound-soft cavities with 5\% noise.  Top row:  $k=3$; Bottom row: $k=5$.}
	\label{fig:cavity_Dirichlet}
\end{figure}
\end{example}

\begin{example}
In the last example, we study the reconstruction of cavities with Neumann boundary condition. The imaging results are given in \cref{fig:cavity_Neumann}. It is worthwhile noting that the sound-hard indicators are in general more sensitive to the noise-contaminated data than the sound-soft ones. In our view, the main reason is the incorporation of additional evaluation on the normal derivatives in the formulation of sound-hard indicators.
\begin{figure}
	\centering	
	\subfigure[]{\includegraphics[width=0.3\textwidth]{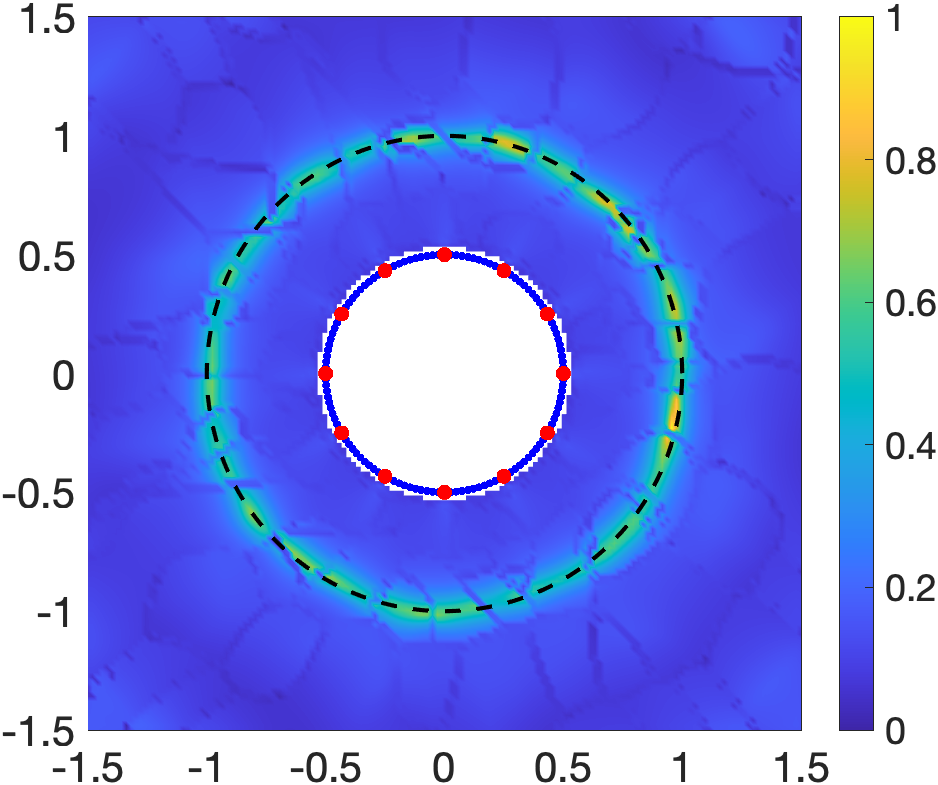}}\quad
	\subfigure[]{\includegraphics[width=0.3\textwidth]{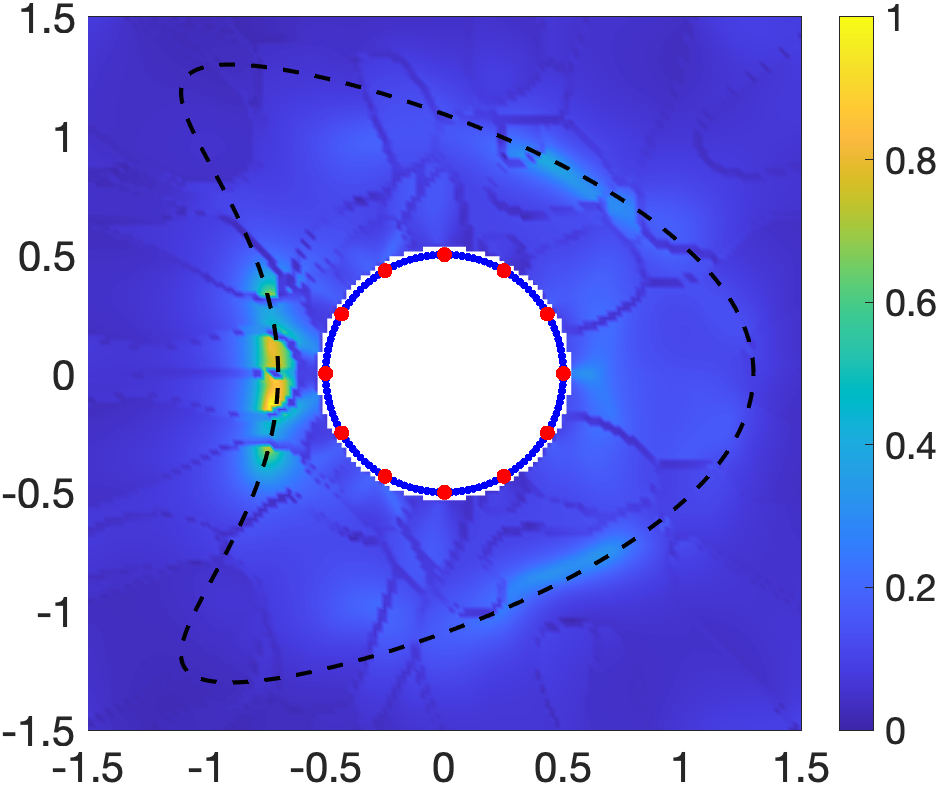}}\quad
	\subfigure[]{\includegraphics[width=0.3\textwidth]{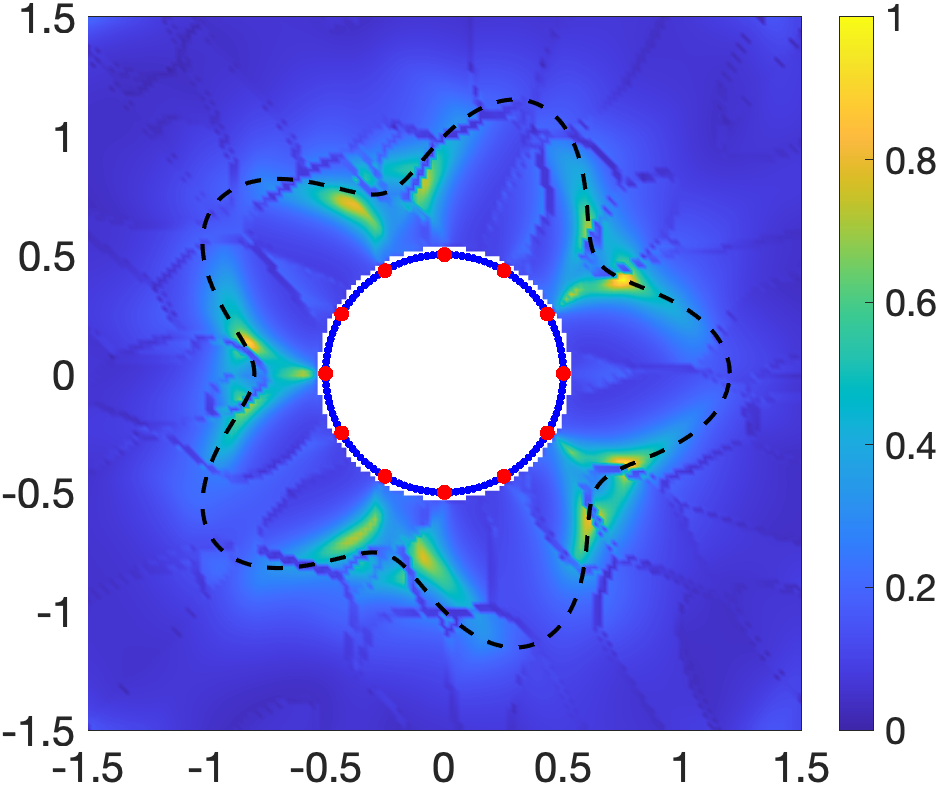}}\\
	\subfigure[]{\includegraphics[width=0.3\textwidth]{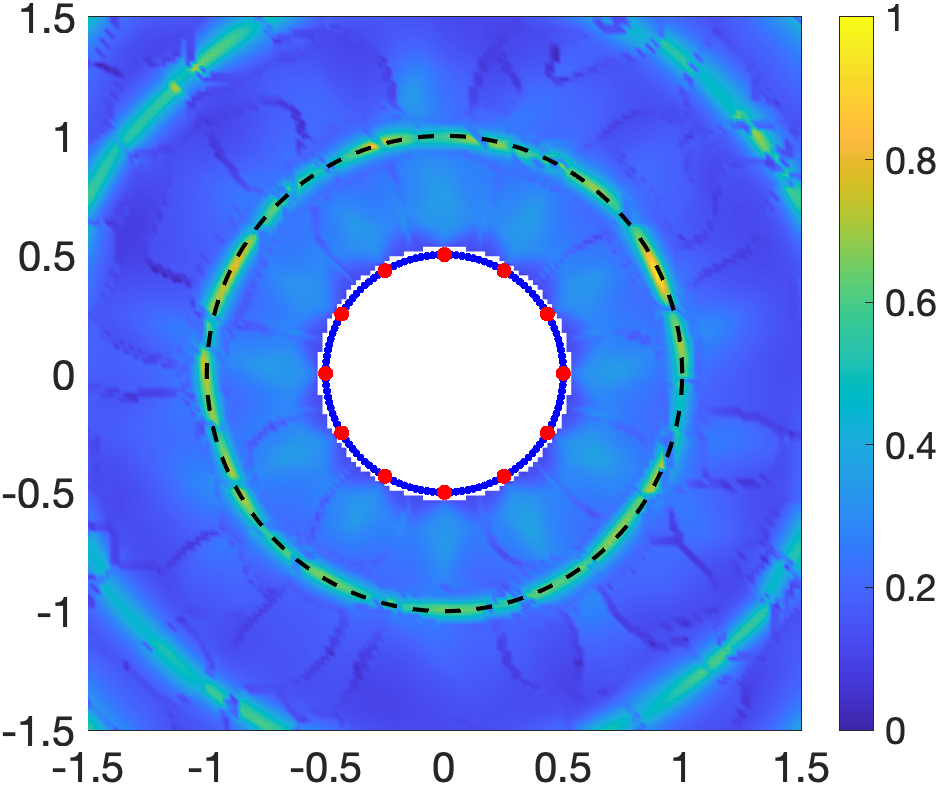}}\quad
	\subfigure[]{\includegraphics[width=0.3\textwidth]{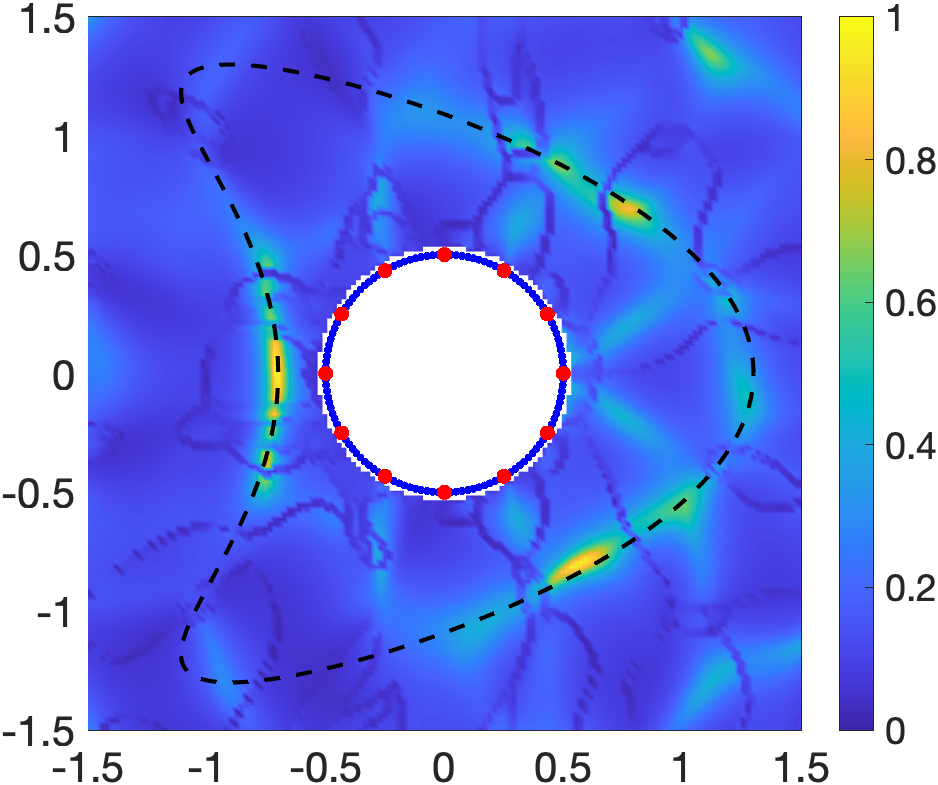}}\quad
	\subfigure[]{\includegraphics[width=0.3\textwidth]{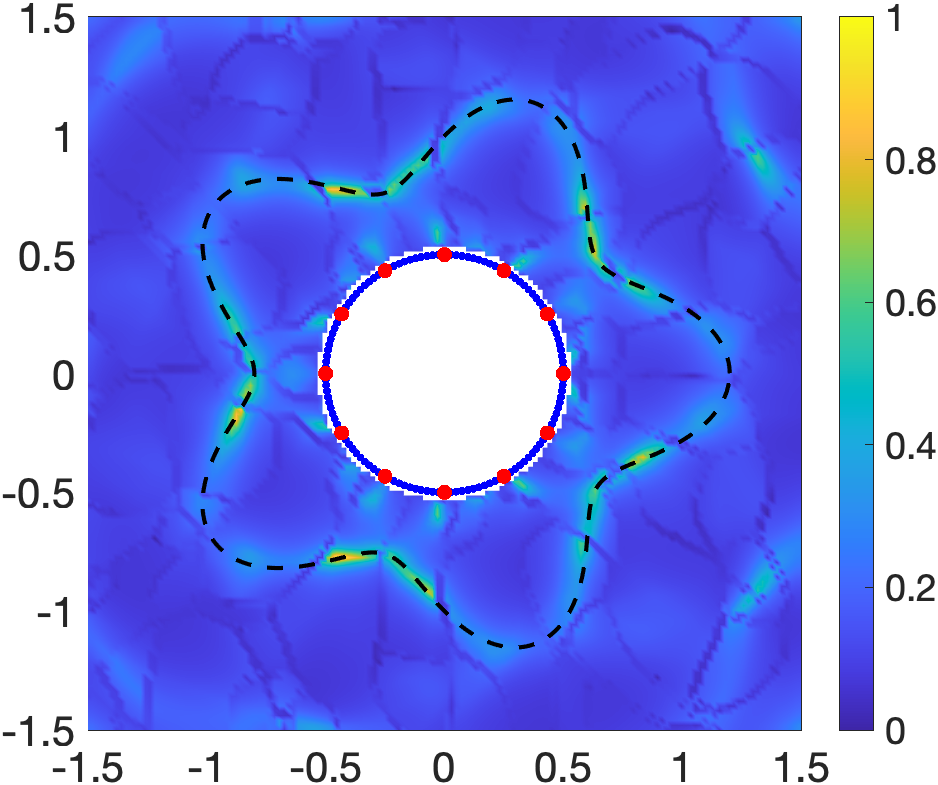}}\\
	\subfigure[]{\includegraphics[width=0.3\textwidth]{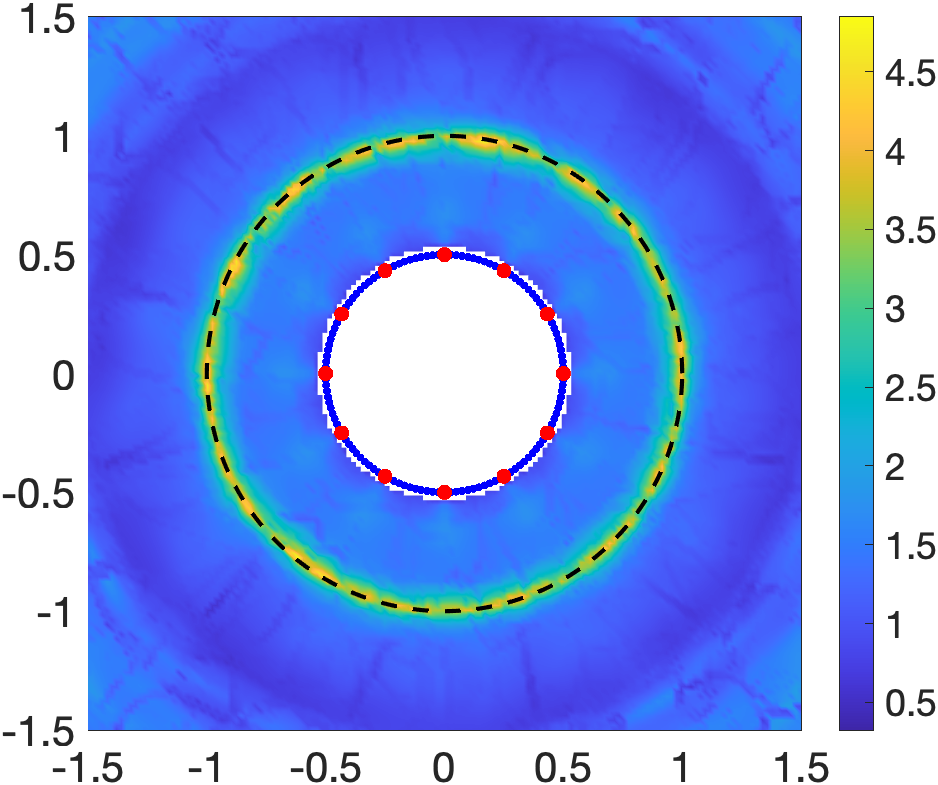}}\quad
	\subfigure[]{\includegraphics[width=0.3\textwidth]{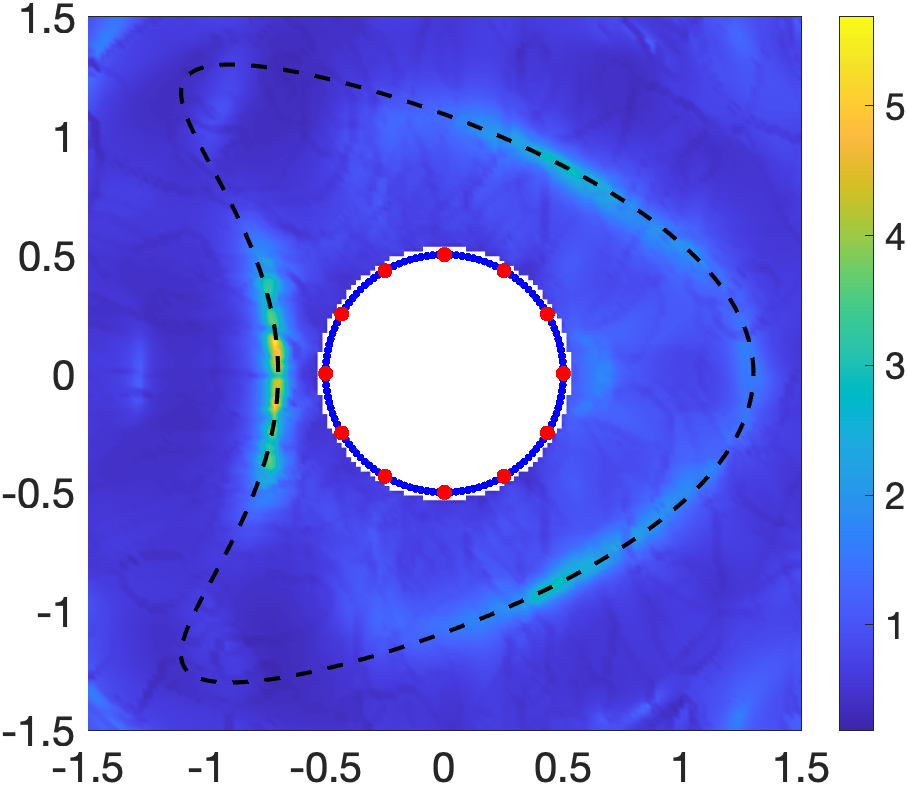}}\quad
	\subfigure[]{\includegraphics[width=0.3\textwidth]{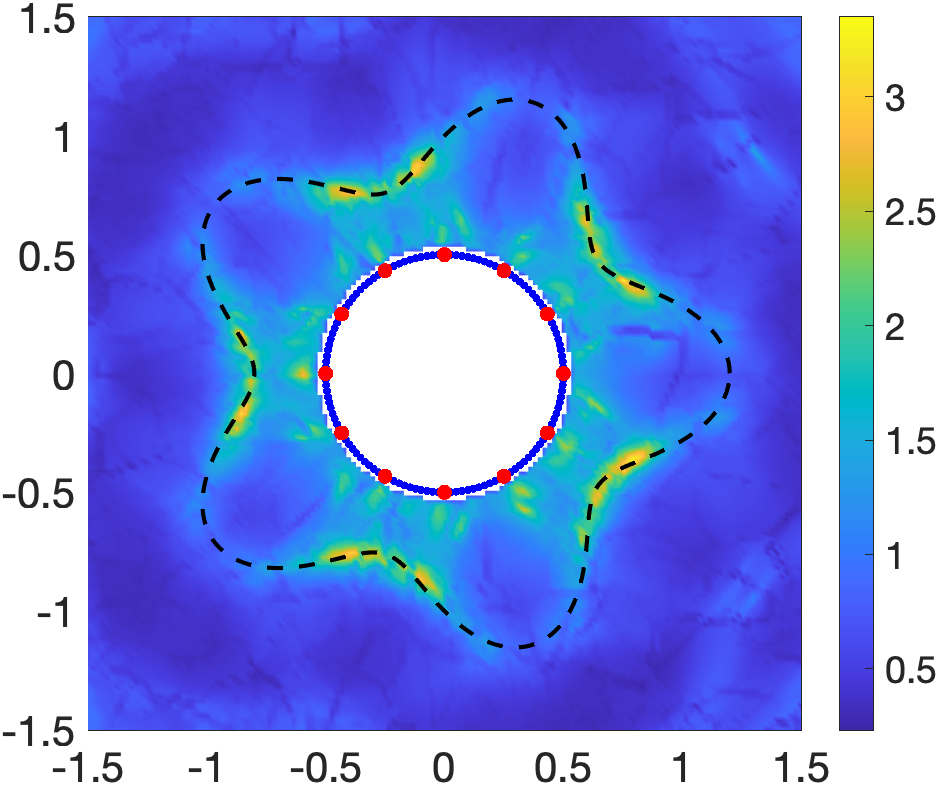}}
	\caption{Reconstruction of sound-hard cavities with 2\% noise.  Top row: $k=4$; Middle row: $k=5$; Bottom row: superposition with multiple frequencies $k=3, 3.5, \cdots, 6$.}
	\label{fig:cavity_Neumann}
\end{figure} 

\end{example}

\section{Conclusion}\label{sec:conclusion}

In this work, we propose a fast computational scheme for solving the inverse acoustic scattering problems. By the a priori sound-soft or sound-hard boundary condition of the scatterer and the Fourier-Bessel approximation of the scattered field, some novel indicator functions are proposed. Then the shape of target obstacle or cavity can be recovered by locating the zeros of associated imaging functions. Theoretical analysis is given to justify the rationale behind the algorithm and simulation experiments are conducted to validate its applicability.

This is an initial attempt in developing a Fourier-Bessel based imaging framework for tackling the inverse scattering problems. Clearly, several computational issues deserve further investigation, for example, is there any optimal choice for the truncation $N$ in the Fourier-Bessel expansion? How to select an appropriate frequency spectrum for a specific inverse scattering problem? Can we improve the accuracy and stability of the reconstructions for sound-hard scatterers? Concerning the future work, feasible extensions include the application of the imaging method to three-dimensional problems or more complicated physical configurations such as mixed boundary conditions. Further extensions to the scenarios of electromagnetic or elastic waves are also interesting topics.

\section*{Acknowledgments}
The work of D. Zhang, Y. Wu and Y. Wang were supported by NSFC grant 12171200. The work of Y. Guo was supported by NSFC grant 11971133 and the Fundamental Research Funds for the Central Universities.

\end{document}